\documentclass{amsart}
\numberwithin{equation}{section}
\usepackage[english]{babel}
\usepackage[T1]{fontenc}
\usepackage[latin1]{inputenc}
\usepackage{indentfirst}
\usepackage{enumitem}
\usepackage{amsmath,amssymb, amsbsy}
\usepackage{amsfonts}
\usepackage{hyperref}
\usepackage{cleveref}
\usepackage{esint}
\usepackage{latexsym}
\usepackage{amsthm}
\usepackage[dvips]{graphicx}
\usepackage{xcolor}
\usepackage{tikz}
\usepackage{pgfplots}
\usetikzlibrary{intersections}
\usepackage{tikz-3dplot}
\usepackage[outline]{contour}
\DeclareGraphicsExtensions{.pdf,.png,.jpg,.eps}

\usepackage{amsthm}

\usepackage{color}
\usepackage[latin1]{inputenc}
\usepackage[active]{srcltx}
\definecolor{Arancio}{cmyk}{0,0.61,0.87,0}

\definecolor{blus}{RGB}{0,102,204}

\newcommand{\brd}[1]{\mathbb{#1}}
\newcommand{\R}{\brd{R}}

\newcommand{\be}{\begin{equation}}
\newcommand{\ee}{\end{equation}}
\newcommand{\loc}{{\text{\tiny{loc}}}}

\newtheorem{teo}{Theorem}[section]
\newtheorem{Corollary}[teo]{Corollary}
\newtheorem{Lemma}[teo]{Lemma}
\newtheorem{Theorem}[teo]{Theorem}
\newtheorem{Proposition}[teo]{Proposition}
\theoremstyle{definition}

\newtheorem{remark}[teo]{Remark}

\usepackage{physics}
\usepackage{amsmath}
\usepackage{tikz}
\usepackage{mathdots}
\usepackage{yhmath}
\usepackage{cancel}
\usepackage{color}
\usepackage{siunitx}
\usepackage{array}
\usepackage{multirow}
\usepackage{amssymb}
\usepackage{gensymb}
\usepackage{tabularx}
\usepackage{extarrows}
\usepackage{booktabs}
\usetikzlibrary{fadings}
\usetikzlibrary{patterns}
\usetikzlibrary{shadows.blur}
\usetikzlibrary{shapes}

\begin{document}

\subjclass[2020] {35B65, 35B44, 35B45, 35B53, 35B30, 35B08}
\keywords{Schauder regularity estimates; Liouville Theorems}

\title[Notes on Schauder estimates by scaling for elliptic PDEs]
{Notes on Schauder estimates by scaling for second order linear elliptic PDEs in divergence form}
\date{\today}

\author{Stefano Vita}

\address[S. Vita]{Dipartimento di Matematica "F. Casorati"
\newline\indent
Universit\`a di Pavia
\newline\indent
Via Ferrata 5, 27100, Pavia, Italy}
\email{stefano.vita@unipv.it}

\maketitle


\begin{abstract}
These are the notes of a part of the PhD course \emph{Regularity for free boundary problems and for elliptic PDEs}, held in Pavia in the spring of 2025. The aim is to provide a comprehensive and self-contained treatment of classical interior and local Schauder estimates for second-order linear elliptic PDEs in divergence form via scaling in the spirit of Simon's work. The main techniques presented here are geometric in nature and were primarily developed in the study of geometric problems such as minimal surfaces. The adopted approach relies on compactness and blow-up arguments, combined with rigidity results (Liouville theorems), and shares many features with the one used in the study of free boundary problems, which was the main topic of the other part of the PhD course.
\end{abstract}

\section{Introduction}
In these notes we are concerned with the local regularity theory for weak solutions to
\begin{equation}\label{E}
-\mathrm{div}(A\nabla u)=f+\mathrm{div}F,\qquad \mathrm{in \ } B_1.
\end{equation}
Here $n\geq 2$ is the space dimension, $B_1=\{x\in \mathbb R^n \, : \, |x|<1\}$ is the unit ball centered at $0$, $u: B_1\to\mathbb R$ is the solution, $f: B_1\to\mathbb R$ is the forcing term, $F=(F_1,...,F_n)$ with $F_i: B_1\to\R$ is a field term and $A=(a_{ij})_{i,j=1,...,n}$ with $a_{ij}: B_1\to\mathbb R$ is the variable coefficient matrix (not necessarily symmetric). In particular the matrix is uniformly elliptic; that is, there exist two constants $0<\lambda\leq\Lambda$ with
\begin{equation}\label{UE}
\lambda |\xi|^2\leq A(x)\xi\cdot\xi\leq\Lambda |\xi|^2,\qquad \mathrm{for \ any \ }x\in B_1, \, \xi\in\mathbb R^n.
\end{equation}

The existence theory for PDEs is set in Sobolev spaces, whose topology is rich enough to allow for minimization of energy functionals. Roughly speaking, $C^k$ spaces are too small to allow an existence theory. However, once the solutions are provided to exist in a weak sense, one would like to promote them to be classic. In the present case of second order equations classic solution means that the partial derivatives up to order two are well defined and the equation is satisfied pointwise.


By interior local regularity, we mean that if the equation holds and some integrability or regularity assumptions on the data are satisfied in the ball $B_1$, then regularity estimates can be obtained for general weak solutions in the smaller ball $B_{1/2}$.

\begin{remark}
There is nothing special about $B_1$ and $B_{1/2}$, which are chosen for the sake of simplicity. The local regularity theory in these notes can be extended to equations in general domains $\Omega\subset\R^n$. The local interior estimates can be obtained in compact subsets $\Omega'\subset\subset\Omega$. This is done by scaling the estimates in the balls and by standard covering arguments of the compact set $\Omega'$. We would like to remark also that some estimates in the present notes could be possibly provided in a scale invariant form. However, this is not a target of the present course, and we will always focus on the local and mainly qualitative information that the estimates imply.
\end{remark}

In these notes, we propose a scheme to derive sharp local (and interior) Schauder estimates based on a regularization-approximation method and blow-up techniques, following Simon's approach in \cite{Sim97}. The main techniques presented here are geometric in nature and were primarily developed in the study of geometric problems such as minimal surfaces and have applications in the study of free boundary problems.

The scheme can be summarized as follows.
In Section \ref{s:2} we introduce the H\"older spaces, the notion of weak solutions and we prove the Caccioppoli inequality. In Section \ref{s:3} we prove $H^2$ estimates using the difference quotients technique by Nirenberg \cite{Nir55} (see also \cite{LioMag72}). Then, we iterate the results on derivatives obtaining $H^k$ estimates for any $k\geq2$. These results imply that weak solutions to \eqref{E} with smooth data, are locally smooth. In Section \ref{s:4} we prove the classic polynomial Liouville theorem for entire harmonic functions, and we obtain local $L^\infty$ bounds for weak solutions with bounded measurable coefficients following the De Giorgi approach \cite{DeG57} (see also \cite{Nas57,Nas58,Mos60,FerRos22,Vas16}). In Section \ref{s:5} we provide a priori $C^{0,\alpha}$ estimates when the coefficients are continuous using a contradiction argument which involves scaling and blow-up procedures in the spirit of Simon's work \cite{Sim97} (see also \cite{GilTru77,HanLin97}). Then, we provide a priori $C^{1,\alpha}$ estimates when the coefficients are $C^{0,\alpha}$ following a similar argument (see \cite{SoaTer19}). In Section \ref{s:6} we imply a posteriori $C^{0,\alpha}$ and $C^{1,\alpha}$ estimates for weak solutions by a regularization-approximation scheme involving convolution of the data with standard mollifiers. Finally, we iterate the $C^{1,\alpha}$ estimate on derivatives obtaining $C^{k,\alpha}$ estimates for any $k\geq 2$.

\begin{remark}
As we will see, the presence of the field term $F$ in the equation \eqref{E} allows us to get general $C^{k,\alpha}$ estimates just iterating a $C^{1,\alpha}$ estimate. Most of the references on Schauder estimates avoid the field term but then need to prove a $C^{2,\alpha}$ estimate for the equation with a $C^{0,\alpha}$ forcing term instead. For the sake of simplicity we decided not to add other lower order terms such as zero order potential terms $Vu$ and first order drift terms $b\cdot\nabla u$. We leave this generalization to the reader.
\end{remark}

Finally, we would like to link the techniques and the results in these notes with free boundary problems, such as obstacle, one phase or two phase problems. On one hand, as already mentioned, our approach to obtaining Schauder estimates relies on compactness and blow-up arguments. This methodology is also central in the analysis of qualitative properties and regularity of solutions of free boundary problems near the free interface. The regularity and measure-theoretic structure of the free boundary itself are typically investigated using similar tools. Regarding the regularity, a common strategy involves first establishing a form of flatness for the regular part of the free boundary, which then implies its Lipschitz continuity and then $C^{1,\alpha}$ regularity. Once this partial regularity is achieved, one can further refine the analysis, often through a bootstrap argument employing Schauder-type estimates, to prove that the regular free boundary is in fact smooth, or even real analytic. We would like to acknowledge some works by S. Salsa and collaborators where this approach has been effectively implemented \cite{FerSal07,FerSal14,DeSFerSal17,DeSFerSal19}.

\section{H\"older spaces, weak solutions, Caccioppoli inequality}\label{s:2}

\subsection{H\"older spaces}
Given $\alpha\in(0,1]$, $C^{0,\alpha}(B_1)$ consists of $C^0(\overline{B_1})$ functions (then uniformly continuous in $B_1$) such that the seminorm
\begin{equation*}
[u]_{C^{0,\alpha}(B_1)}=\sup_{\substack{x,y\in B_1 \\ x\neq y}}\frac{|u(x)-u(y)|}{|x-y|^\alpha}<\infty.
\end{equation*}
The $C^{0,\alpha}$-norm is defined as
\begin{equation*}
\|u\|_{C^{0,\alpha}(B_1)}=\|u\|_{L^\infty(B_1)}+[u]_{C^{0,\alpha}(B_1)}.
\end{equation*}
Notice that $\alpha=1$ corresponds to Lipschitz continuous functions. Given $k\in\mathbb N$, $\alpha\in(0,1]$, $C^{k,\alpha}(B_1)$ consists of $C^k(\overline{B_1})$ functions (then partial derivatives up to order $k$ are uniformly continuous in $B_1$) such that the seminorm
\begin{equation*}
[D^\beta u]_{C^{0,\alpha}(B_1)}<\infty,
\end{equation*}
where $\beta=(\beta_1,...,\beta_n)\in\mathbb N^n$ is any multiindex with $|\beta|=\sum_{i=1}^{n}\beta_i=k$.
The $C^{0,\alpha}$-norm is defined as
\begin{equation*}
\|u\|_{C^{k,\alpha}(B_1)}=\sum_{i=0}^k\sum_{|\beta|=i}\|D^\beta u\|_{L^\infty(B_1)}+\sum_{|\beta|=k}[D^\beta u]_{C^{0,\alpha}(B_1)}.
\end{equation*}
For simplicity, we will indicate by $D^ku$ a generic partial derivative of order $k$; that is, $D^\beta u$ with $|\beta|=k$.

\begin{remark}
It is easy to see that $C^{k,\alpha}(B_1)=C^{k,\alpha}(\overline{B_1})$ (when $\alpha\in(0,1]$) since the uniform continuity on a set or on its topological closure are equivalent. Moreover, one has
\begin{equation*}
C^0(\overline{B_1}) \supset C^{0,\alpha}(B_1) \supset C^{0,1}(B_1) \supset C^1(\overline{B_1}) \supset C^{1,\alpha}(B_1) \supset ... \supset C^\infty(\overline{B_1}).
\end{equation*}
\end{remark}

\subsection{Weak solutions}
A weak solution of \eqref{E} in $B_1$ is a function $u\in H^1(B_1)$ such that
\begin{equation*}
\int_{B_1}A\nabla u\cdot\nabla\phi=\int_{B_1}f\phi-\int_{B_1}F\cdot\nabla\phi\qquad\mathrm{for \ any \ }\phi\in H^1_0(B_1).
\end{equation*}
By density one can equivalently test the above equation against any $\phi\in C^\infty_c(B_1)$. Let us recall here the uniform ellipticity conditions in \eqref{UE}. Since we are always assuming that coefficients are bounded measurable; that is, their $L^\infty$ norm is bounded, we will assume the existence of $L>0$ such that
\begin{equation}\label{UE2}
\|A\|_{L^\infty(B_1)}\leq L.
\end{equation}
Notice that the bound from above in \eqref{UE} (the one involving $\Lambda$) is implied by the strongest condition \eqref{UE2}, since for any $\xi_1,\xi_2\in\R^n$
\begin{equation}\label{bilinear}
A\xi_1\cdot\xi_2\leq |A\xi_1||\xi_2|\leq \|A\|_{op}|\xi_1||\xi_2|\leq n\|A\|_{L^\infty(B_1)}|\xi_1||\xi_2|,
\end{equation}
where $\|A\|_{op}=\sup_{|\xi|=1}|A\xi|$. We also remark that in case of symmetric coefficients, the upper bound in \eqref{UE} implies \eqref{bilinear} without assuming  \eqref{UE2}, since $\|A\|_{op}=\Lambda$. We say that a constant $C_0>0$ is universal in $B_1$ if it depends only on the dimension $n$ and on the ellipticity constants $\lambda,\Lambda,L$ in $B_1$.

\subsection{Caccioppoli inequality}
The following inequality is a main tool for the regularity estimates.
\begin{Proposition}[Caccioppoli inequality]\label{p:caccioppoli}
Let $0<r<R\leq1$. Then there exists a universal constant $C>0$ in $B_1$ such that
\begin{equation*}
\|\nabla u\|_{L^2(B_r)}\leq C\left(\frac{1}{R-r}\|u\|_{L^2(B_R)}+\|f\|_{L^2(B_R)}+\|F\|_{L^2(B_R)}\right)
\end{equation*}
for any weak solution $u$ to \eqref{E} in $B_1$.
\end{Proposition}
\begin{proof}
Let $\eta\in C^\infty_c(B_R)$ with $0\leq\eta\leq1$, radially decreasing cut-off function with $\eta=1$ in $B_r$. Such a function can be chosen such that $|\nabla\eta|\leq2(R-r)^{-1}$. We test \eqref{E} with $\eta^2u\in H^1_0(B_1)$; that is,
\begin{equation*}
\int_{B_1}A\nabla u\cdot\nabla(\eta^2u)=\int_{B_1}f\eta^2u-F\cdot\nabla(\eta^2u).
\end{equation*}
Then
\begin{eqnarray*}
A\nabla u\cdot\nabla(\eta^2u)&=&\eta A\nabla u\cdot\nabla(\eta u)+\eta u A\nabla u\cdot\nabla\eta\\
&=&A\nabla(\eta u)\cdot\nabla(\eta u)-uA\nabla\eta\cdot\nabla(\eta u)+uA\nabla(\eta u)\cdot\nabla\eta-u^2A\nabla\eta\cdot\nabla\eta.
\end{eqnarray*}
Hence
\begin{eqnarray*}
\int_{B_1}A\nabla(\eta u)\cdot\nabla(\eta u)&\leq&\int_{B_1}|uA\nabla\eta\cdot\nabla(\eta u)|+\int_{B_1}|uA\nabla(\eta u)\cdot\nabla\eta|+\int_{B_1}|u^2A\nabla\eta\cdot\nabla\eta|\\
&&+\int_{B_1}|f\eta^2 u|+\int_{B_1}|\eta F\cdot\nabla(\eta u)|+\int_{B_1}|\eta u F\cdot\nabla\eta|.
\end{eqnarray*}
By the Young inequality with a chosen $\varepsilon>0$ to be announced, \eqref{UE} and \eqref{UE2}, we get
\begin{eqnarray*}
\lambda\int_{B_1}|\nabla(\eta u)|^2&\leq&\varepsilon\int_{B_1}|\nabla(\eta u)|^2+\frac{n^2L^2}{\varepsilon}\int_{B_1}u^2|\nabla\eta|^2+\Lambda\int_{B_1}u^2|\nabla\eta|^2+\frac{1}{2}\|f\|_{L^2(B_R)}^2\\
&&+\frac{1}{2}\|u\|_{L^2(B_R)}^2+\frac{1}{2}\left(1+\frac{1}{\varepsilon}\right)\|F\|_{L^2(B_R)}^2+\frac{\varepsilon}{2}\int_{B_1}|\nabla(\eta u)|^2+\frac{1}{2}\int_{B_1}u^2|\nabla\eta|^2.
\end{eqnarray*}
Then, using $|\nabla\eta|\leq2(R-r)^{-1}$ and $\eta=1$ in $B_r$, there exists a universal constant $C>0$ such that
\begin{equation*}
\left(\lambda-\varepsilon-\frac{\varepsilon}{2}\right)\int_{B_r}|\nabla u|^2\leq C\left(\frac{1}{R-r}\|u\|_{L^2(B_R)}+\|f\|_{L^2(B_R)}+\|F\|_{L^2(B_R)}\right)^2.
\end{equation*}
The result follows by choosing $\varepsilon=\lambda/2$ and taking the square roots in the above inequality. We remark that the constant depends on the ellipticity ratio $\Lambda/\lambda$ and also on $nL/\Lambda$. The latter can be chosen to be $1$ in the symmetric case. 
\end{proof}

In the result above we considered the equation \eqref{E} satisfied in $B_1$ for simplicity. The same result holds if the equation is satisfied in a ball $B_{\overline R}$ and considering $0<r<R\leq\overline R$ but with a constant that depends on $\max\{1,R\}$ too in case of nontrivial right hand sides.

\section{$H^2$ estimates, $H^k$ estimates, smooth data imply smooth solutions}\label{s:3}

\subsection{$H^2$ estimates}
In this section we revisit the classic $H^2$ interior regularity estimate for weak solutions to
\begin{equation}\label{E2}
-\mathrm{div}(A\nabla u)=\mathrm{div}F\qquad\mathrm{in \ }B_1.
\end{equation}
Let us recall that, given $k\geq1$ the space $H^k(B_1)$ stands for $W^{k,2}(B_1)$. This means that the weak partial derivatives $D^ju\in L^2(B_1)$ for any $j=0,...,k$. For simplicity (and this is sufficient for our purposes) we do not deal with forcing terms in the following result.
\begin{Theorem}[$H^2$ estimates]\label{t:H2}
Let $A\in C^{0,1}(B_1)$ with $\|A\|_{C^{0,1}(B_1)}\leq \overline L$, $F\in H^1(B_1)$ and $u\in H^1(B_1)$ be a weak solution to \eqref{E2}. Then $u\in H^2(B_{1/2})$ and there exists a constant $C>0$ depending only on $n$, the ellipticity constants and $\overline L$ such that
\begin{equation*}
\|u\|_{H^2(B_{1/2})}\leq C(\|u\|_{L^2(B_1)}+\|F\|_{H^1(B_1)}).
\end{equation*}
\end{Theorem}
In order to prove the above result, we make use of the difference quotients technique introduced by Nirenberg \cite{Nir55}. The incremental quotient of step $h\neq0$ and direction $e_j$ with $j\in\{1,...,n\}$ is given by
\begin{equation*}
D_j^hu(x):=\frac{u(x+he_j)-u(x)}{h}.
\end{equation*}
The following Lemma states the main properties of the incremental quotients and the proof is omitted (see for instance \cite[Section 7.11]{GilTru77}).
\begin{Lemma}\label{l:difference}
Let $u\in H^1(B_R)$, $0<r<R$, $0<|h|<R-r$, $i,j\in\{1,...,n\}$. Then
\begin{itemize}
\item[(i)] $\|D^h_ju\|_{L^2(B_r)}\leq \frac{2}{|h|}\|u\|_{L^2(B_R)};$
\item[(ii)] For any $\phi\in C^\infty_c(B_r)$
$$\int_{B_R}D_j^hu\phi=-\int_{B_R}uD_j^{-h}\phi;$$
\item[(iii)] $\|D^h_ju\|_{L^2(B_r)}\leq \|\partial_j u\|_{L^2(B_R)};$
moreover, up to subsequences, $D^h_ju\rightharpoonup \partial_ju$ in $L^2(B_r)$;
\item[(iv)] $\partial_i (D^h_ju)=D^h_j(\partial_iu)$ and $D_j^{h}u\in H^1(B_r)$.
\end{itemize}
\end{Lemma}

\begin{proof}[Proof of Theorem \ref{t:H2}]
Let us consider $\phi\in C^\infty_c(B_{3/4})\subset C^\infty_c(B_{1})$ and test \eqref{E2} against $\phi$; that is, 
\begin{equation}\label{e:phi}
-\int_{\mathrm{supp}\phi\subset B_1}A(x)\nabla u(x)\cdot\nabla\phi(x)=\int_{\mathrm{supp}\phi\subset B_1}F(x)\cdot\nabla\phi(x).
\end{equation}

Then given $j\in\{1,...,n\}$ and $0<|h|<1/8$, let us consider $\phi(\cdot-he_j)$ which belongs to $C^\infty_c(B_{3/4}(he_j))\subset C^\infty_c(B_1)$. Then, we can test \eqref{E2} also against $\phi(\cdot-he_j)$; that is,
\begin{equation*}
-\int_{\mathrm{supp}\phi(\cdot-he_j)\subset B_1}A(x')\nabla u(x')\cdot\nabla\phi(x'-he_j)=\int_{\mathrm{supp}\phi(\cdot-he_j)\subset B_1}F(x')\cdot\nabla\phi(x'-he_j),
\end{equation*}
and after a change of variable $x=x'-he_j$, this leads to
\begin{equation}\label{e:phi2}
-\int_{\mathrm{supp}\phi\subset B_1}A(x+he_j)\nabla u(x+he_j)\cdot\nabla\phi(x)=\int_{\mathrm{supp}\phi\subset B_1}F(x+he_j)\cdot\nabla\phi(x).
\end{equation}
Subtracting \eqref{e:phi2} and \eqref{e:phi}, and dividing by $h$ we get
\begin{equation*}\label{e:incremental}
-\int_{B_1}A(x+he_j)D^h_j(\nabla u)(x)\cdot\nabla\phi(x)=\int_{B_1}D_j^hA(x)\nabla u(x)\cdot\nabla\phi(x)+\int_{B_1}D^h_jF(x)\cdot\nabla\phi(x).
\end{equation*}
Using that $D^h_j(\nabla u)=\nabla (D^h_ju)$ and that $D_j^hu\in H^1(B_{3/4})$ (point (iv) of Lemma \ref{l:difference}), the above formulation, which holds true for any $\phi\in C^\infty_c(B_{3/4})$ and for $0<|h|<1/8$ says that $D_j^hu$ weakly solves
\begin{equation*}
-\mathrm{div}(A(\cdot+he_j)\nabla (D_j^hu))=\mathrm{div}(D_j^hA\nabla u+D^h_jF)\qquad\mathrm{in \ }B_{3/4}.
\end{equation*}
Then, given $0<r<R\leq3/4$ the Caccioppoli inequality in Proposition \ref{p:caccioppoli} says that
\begin{equation*}
\|\nabla(D_j^hu)\|_{L^2(B_r)}\leq c\left(\frac{1}{R-r}\|D_j^hu\|_{L^2(B_R)}+\|D_j^hF+D_j^hA\nabla u\|_{L^2(B_R)}\right).
\end{equation*}
Then, point (iii) of Lemma \ref{l:difference}, together with the condition $|h|<1/8$ says that
\begin{equation*}
\|D_j^hu\|_{L^2(B_R)}\leq \|\partial_ju\|_{L^2(B_{7/8})},
\end{equation*}
which in turns is estimated by the Caccioppoli inequality on the equation for $u$ itself; that is,
\begin{equation*}
\|\partial_ju\|_{L^2(B_{7/8})}\leq \|\nabla u\|_{L^2(B_{7/8})}\leq C\left(\|u\|_{L^2(B_1)}+\|F\|_{L^2(B_1)}\right).
\end{equation*}
Moreover, using again point (iii) of Lemma \ref{l:difference} and the Caccioppoli inequality for $u$
\begin{eqnarray*}
\int_{B_R}|D_j^hF+D_j^hA\nabla u|^2&\leq&2\int_{B_R}|D_j^hF|^2+2\int_{B_R}|D_j^hA\nabla u|^2\\
&\leq&2\int_{B_1}|\partial_jF|^2+2n^2\sup_{x\in B_R}|D_j^hA|^2\int_{B_R}|\nabla u|^2\\
&\leq&2\int_{B_1}|\nabla F|^2+2n^2\overline L^2C(\|u\|_{L^2(B_1)}+\|F\|_{L^2(B_1)})^2
\end{eqnarray*}
We used the Lipschitz continuity of coefficients; that is, $|a_{ij}(x+he_j)-a_{ij}(x)|\leq \overline L |h|$. This allows us to infer the bound
\begin{equation*}
\|D_j^hF+D_j^hA\nabla u\|_{L^2(B_R)}\leq C(\|u\|_{L^2(B_1)}+\|F\|_{H^1(B_1)}).
\end{equation*}
Summing together the information obtained, we have the existence of a constant which depends on the bound on the $C^{0,1}$-norm of coefficients such that
\begin{equation*}
\|\nabla(D_j^hu)\|_{L^2(B_r)}\leq C\left(\|u\|_{L^2(B_1)}+\|F\|_{H^1(B_1)}\right).
\end{equation*}
Hence $\nabla(D^h_j u)=D_j^h(\nabla u)$ is uniformly bounded in $L^2(B_r)$ in $|h|<1/8$ (i.e. $D^h_j u$ is uniformly bounded in $H^1(B_r)$). Hence, it weakly converges in $L^2(B_r)$ to $\nabla(\partial_j u)=\partial_j(\nabla u)$ (using point (ii) of Lemma \ref{l:difference}). This gives the belonging to $H^2(B_{1/2})$ by choosing $r=1/2$, and the desired estimate using the lower semicontinuity of the $L^2$-norm with respect to the weak convergence
\begin{equation*}
\|\nabla(\partial_ju)\|_{L^2(B_r)}\leq \liminf_{h\to0}\|\nabla(D_j^hu)\|_{L^2(B_{1/2})}\leq C\left(\|u\|_{L^2(B_1)}+\|F\|_{H^1(B_1)}\right).
\end{equation*}
\end{proof}

\begin{remark}\label{r:H2}
From the proof of the previous result it is clear that having the equation satisfied on a ball $B_R$ the estimate is available on any smaller ball $B_r$ (i.e. $0<r<R$)
\begin{equation*}
\|u\|_{H^2(B_{r})}\leq C(\|u\|_{L^2(B_R)}+\|F\|_{H^1(B_R)}).
\end{equation*}
with a constant that may depend on both $R$ and $r$ and explodes if $R-r\to0$. More precisely, if $R\leq1$ then the constant $C=\overline C/(R-r)^2$ where $\overline C$ is universal in $B_1$.
\end{remark}

\begin{remark}[Scaling and covering]\label{r:scaling-covering}
Imagine to have proven an estimate $B_1\to B_r$ for solutions to \eqref{E2} in $B_1$ with $0<r<1$; that is, there exists a constant (which depends on $\|A\|_{C^{0,1}(B_1)}\leq \overline L$ and explodes as $1-r\to0$) such that
\begin{equation*}
\|u\|_{H^2(B_{r})}\leq C(\|u\|_{L^2(B_1)}+\|F\|_{H^1(B_1)}).
\end{equation*}
Then, from this one can get an estimate $B_1\to B_{R}$ with $0<r<R<1$. One can procede as follows. 
\begin{itemize}
\item[(i)] From the estimate $B_{1}\to B_{r}$ one can get an estimate $B_{t}(x_0)\to B_{t r}(x_0)$ for any $t\in(0,1)$ and any $x_0\in B_1$ such that $B_{t}(x_0)\subset B_1$. Such an estimate is $t$-dependent, and it is obtained by considering a given solution $u$ of \eqref{E2} in $B_{t}(x_0)$ and scaling it to $v(x)=u(x_0+t x)$, which is a solution in $B_1$ to 
\begin{equation*}
-\mathrm{div}(\tilde A\nabla v)=\mathrm{div} \tilde F\qquad \mathrm{in \ }B_1,
\end{equation*}
where $\tilde A(x)=A(x_0+t x)$, $\tilde F(x)=\lambda F(x_0+t x)$. Then, the estimate $B_1\to B_r$ says that there exists a constant (which depends on $\|\tilde A\|_{C^{0,1}(B_1)}\leq \|A\|_{C^{0,1}(B_1)} \leq \overline L$ and explodes as $1-r\to0$) such that
\begin{equation*}
\|v\|_{H^2(B_{r})}\leq C(\|v\|_{L^2(B_1)}+\|\tilde F\|_{H^1(B_1)}).
\end{equation*}
This gives
\begin{equation*}
\|u\|_{H^2(B_{t r}(x_0))}\leq \frac{C}{t^2}(\|u\|_{L^2(B_1)}+\|F\|_{H^1(B_1)}).
\end{equation*}
\item[(ii)] Then, in order to prove the estimate in $B_{1}\to B_{R}$ one can proceed by a covering argument. Choose $t>0$ small enough so that $t<1-R$. Then one can cover $\overline{B_{R}}$ with a finite number of balls of radius $t r$ centered at points of $\overline{B_{R}}$; that is,
\begin{equation*}
\overline{B_{R}}\subset\bigcup_{i=1}^NB_{tr}(x_i)\subset \bigcup_{i=1}^NB_{t}(x_i)\subset B_{1}.
\end{equation*}
\end{itemize}
\end{remark}

\begin{Corollary}\label{c:derivative}
Under the hypothesis of Theorem \ref{t:H2}, for any $i\in\{1,...,n\}$ and any fixed $0<r<1$ we have that $u_i=\partial_iu\in H^1(B_r)$ is a weak solution to 
\begin{equation*}
-\mathrm{div}(A\nabla u_i)=\mathrm{div}(\partial_i A\nabla u+\partial_iF) \qquad\mathrm{in \ } B_{r}.
\end{equation*}
\end{Corollary}

\begin{proof}
The fact that $u_i\in H^1(B_r)$ is implied by $u\in H^2(B_r)$. Then, in order to have the weak formulation for $u_i$ we need to pass to the limit in the weak formulation for $D^h_iu$; that is,
\begin{equation*}
-\int_{B_1}A(x+he_i)D^h_i(\nabla u)(x)\cdot\nabla\phi(x)=\int_{B_1}D_i^hA(x)\nabla u(x)\cdot\nabla\phi(x)+\int_{B_1}D^h_iF(x)\cdot\nabla\phi(x),
\end{equation*}
which holds true for any $\phi\in C^\infty_c(B_r)$ just taking $|h|<<1$ small enough. Then one can rewrite the formulation above as
\begin{eqnarray*}
-\int_{B_1}A(x)\nabla (D^h_i u)(x)\cdot\nabla\phi(x)&=&\int_{B_1}(A(x+he_i)-A(x))\nabla(D^h_i u)(x)\cdot\nabla\phi(x)\\
&&+\int_{B_1}D_i^hA(x)\nabla u(x)\cdot\nabla\phi(x)+\int_{B_1}D^h_iF(x)\cdot\nabla\phi(x).
\end{eqnarray*}
Using that $D^h_i(\nabla u)=\nabla(D^h_iu)$ is uniformly bounded in $L^2$ (i.e. $D^h_iu$ is uniformly bounded in $H^1$) and hence the weak convergence in $H^1$ one has 
$$\int_{B_1}A(x)\nabla (D^h_i u)(x)\cdot\nabla\phi(x)\to\int_{B_1}A(x)\nabla u_i(x)\cdot\nabla\phi(x).$$
Again by the uniform bound in $H^1$ of $D^h_iu$ and the uniform continuity of $A$, we get
$$\int_{B_1}(A(x+he_i)-A(x))\nabla(D^h_i u)(x)\cdot\nabla\phi(x)\to0.$$
The Lipschitz continuity of $A$ gives a.e. differentiability and hence
$$\int_{B_1}D_i^hA(x)\nabla u(x)\cdot\nabla\phi(x)\to\int_{B_1}\partial_iA(x)\nabla u(x)\cdot\nabla\phi(x).$$
Finally, using point (iii) in Lemma \ref{l:difference}; that is, the weak convergence $D^h_iF\to\partial_iF$ in $L^2$, we have
$$\int_{B_1}D^h_iF(x)\cdot\nabla\phi(x)\to\int_{B_1}\partial_iF(x)\cdot\nabla\phi(x).$$
\end{proof}

\subsection{$H^k$ estimates}
Below we state the $H^k$ local regularity theorem for general weak solutions to \eqref{E2} in $B_R$. Here $k\geq2$ and so Theorem \ref{t:H2} is included. It is more convenient to state the result with general radii $0<r<R$ since the induction argument involved in its proof requires the estimate in general balls centered at $0$.
\begin{Theorem}[$H^k$ estimates]\label{t:Hk}
Let $k\geq2$, $R>0$, $A\in C^{k-2,1}(B_R)$ with $\|A\|_{C^{k-2,1}(B_R)}\leq \overline L$, $F\in H^{k-1}(B_R)$ and $u\in H^1(B_R)$ be a weak solution to \eqref{E2} in $B_R$. Then $u\in H^{k}_\loc(B_{R})$ and given $0<r<R$ there exists a constant $C>0$ depending only on $n$, $k$, the ellipticity constants, $\overline L$ and $R,r>0$ (blows-up as $R-r\to0$) such that
\begin{equation*}
\|u\|_{H^{k}(B_{r})}\leq C(\|u\|_{L^2(B_R)}+\|F\|_{H^{k-1}(B_R)}).
\end{equation*}
\end{Theorem}
\begin{proof}
Let us prove the result by induction on $k\geq2$. The case $k=2$ is Theorem \ref{t:H2} (together with Remark \ref{r:H2}). Then let us suppose the result true for a general $k\geq2$ and prove it for $k+1$. So, $A\in C^{k-1,1}(B_1)$, $F\in H^k(B_1)$ and we want to prove that the weak solution $u\in H^1(B_1)$ actually belongs to $H^{k+1}(B_{1/2})$ together with the estimate in $B_{1/2}$ (wlog we can choose $R=1$ and $r=1/2$ for the sake of simplicity). Let us consider a given partial derivative $u_i=\partial_i u$ with $i\in\{1,...,n\}$. Theorem \ref{t:H2} and Corollary \ref{c:derivative} are saying that $u_i\in H^1(B_{r})$ for any $0<r<1$ and is a solution to
\begin{equation*}
-\mathrm{div}(A\nabla u_i)=\mathrm{div}(\partial_i A\nabla u+\partial_iF) \qquad\mathrm{in \ } B_{r}.
\end{equation*}
By our assumptions we know that $\partial_iA\in C^{k-2,1}(B_1)$, $\partial_iF\in H^{k-1}(B_1)$ and by the inductive hypothesis we also know that $\nabla u\in H^{k-1}(B_{r})$. Then, $u_i$ is a solution of
\begin{equation*}
-\mathrm{div}(A\nabla u_i)=\mathrm{div}(\tilde F) \qquad\mathrm{in \ } B_{r},
\end{equation*}
with $\tilde F:=\partial_i A\nabla u+\partial_iF\in H^{k-1}(B_{r})$ (by induction again it is easy to see that the product $\partial_i A\nabla u\in H^{k-1}$). Then, the inductive hypothesis again gives us the desired regularity $u_i\in H^k(B_{1/2})$ with the estimate
\begin{equation*}
\|u_i\|_{H^k(B_{1/2})}\leq C(\|u_i\|_{L^2(B_{r})}+\|\tilde F\|_{H^{k-1}(B_{r})}),
\end{equation*}
which can be easily manipulated to the desired one having
$$\|\partial_i A\nabla u\|_{H^{k-1}(B_{r})}\leq \|\partial_i A\|_{C^{k-2,1}(B_{r})}\|\nabla u\|_{H^{k-1}(B_{r})},$$
and applying again the inductive hypothesis together with the Caccioppoli inequality for $u$.
\end{proof}

\subsection{Smooth data imply smooth solutions}
$H^k$ estimates for solutions to \eqref{E2} imply $C^\infty$ regularity for solutions to \eqref{E} when data are smooth.
\begin{Corollary}[Smooth data $\Rightarrow$ smooth solutions]\label{c:smooth}
Let $A,f,F\in C^\infty(B_1)$ and $u\in H^1(B_1)$ be a weak solution to \eqref{E} in $B_1$. Then $u\in C^\infty_\loc(B_1)$.
\end{Corollary}
\begin{proof}
We can rewrite the equation \eqref{E} as \eqref{E2} for a certain $\tilde F\in C^\infty(B_1)$; that is,
\begin{equation*}
-\mathrm{div}(A\nabla u_i)=\mathrm{div}(\tilde F) \qquad\mathrm{in \ } B_{1}.
\end{equation*}
In particular $A\in C^{k-2,1}(B_1)$ and $\tilde F\in H^{k-1}(B_1)$ for any given $k\geq2$. Then, Theorem \ref{t:Hk} implies $u\in H^k_\loc(B_1)$ for any $k\geq2$, which leads to $u\in C^\infty_\loc(B_1)$ via Morrey embeddings.
\end{proof}

\begin{remark}
A way to pass from a forcing term to the divergence of a field term is the following:
\begin{equation*}
f(x',x_n)=\partial_n\left(\int_0^{x_n}f(x',t)dt\right)=\mathrm{div}\left(e_n\int_0^{x_n}f(x',t)dt\right)=\mathrm{div} F(x',x_n).
\end{equation*}
If $f\in C^\infty(B_1)$, then also $F\in C^\infty(B_1)$.
\end{remark}

\section{Liouville theorem, $L^\infty$ bounds}\label{s:4}

\subsection{Liouville theorem}
We say that $u\in H^1_\loc(\R^n)$ (i.e. $H^1(B_R)$ for any $R>0$) is an entire harmonic function, and we write $-\Delta u=0$ in $\R^n$, if
\begin{equation*}
\int_{\R^n}\nabla u\cdot\nabla\phi=0\qquad\mathrm{for \ any \ }\phi\in C^\infty_c(\R^n).
\end{equation*}
Before stating the Liouville theorem we recall the Caccioppoli inequality in case of zero right hand sides
\begin{Proposition}[Caccioppoli inequality with zero right hand sides]\label{p:caccioppoli2}
Let $0<r<R$. Then there exists a universal constant $C>0$ in $B_R$ such that
\begin{equation*}
\|\nabla u\|_{L^2(B_r)}\leq \frac{C}{R-r}\|u\|_{L^2(B_R)}
\end{equation*}
for any weak solution $u$ to \eqref{E} in $B_R$ with $f=F=0$.
\end{Proposition}

The following is a rigidity result which states that the only entire harmonic functions having a polynomial growth are the harmonic polynomials.

\begin{Theorem}[Liouville]\label{t:liouville}
Let $u$ be an entire harmonic function in $\R^n$ such that there exist two constants $C>0$ and $\gamma\geq0$ such that
\begin{equation*}
|u(x)|\leq C(1+|x|)^\gamma,\qquad\mathrm{for \ any \ }x\in\R^n.
\end{equation*}
Then $u$ is a polynomial of degree at most $\lfloor \gamma\rfloor:=\max\{k\in\mathbb Z \, : \, k\leq\gamma\}$.
\end{Theorem}
\begin{proof}
We already know that $u\in C^\infty(\R^n)$ by Corollary \ref{c:smooth} and that $D^ku$ is still entire harmonic in $\R^n$ for any $k\in\mathbb N$ by Corollary \ref{c:derivative}. Given an arbitrary ball $B_R$ with $R>0$, applying repeatedly the Caccioppoli inequality in Proposition \ref{p:caccioppoli2} one can get
\begin{equation*}
\int_{B_R}u^2\geq c_1R^2\int_{B_{R/2}}|\nabla u|^2\geq c_1R^2\int_{B_{R/2}}|D u|^2 \geq c_2R^4\int_{B_{R/4}}|\nabla (Du)|^2\geq c_2R^4\int_{B_{R/4}}|D^2u|^2\geq...
\end{equation*}
So, given any partial derivative $D^ku$ of order $k$, using the polynomial growth condition on $u$
\begin{equation*}
c_kR^{2k}\int_{B_{R/2^k}}|D^ku|^2\leq CR^{2\gamma+n}.
\end{equation*}
Let us take $k\in\mathbb N$ so that $2\gamma+n-2k<0$. 
Hence taking an arbitrary compact set $K\subset\R^n$, one has $K\subset B_{R/2^k}$ for $R>0$ large enough, and considering the limit $R\to\infty$, we get
\begin{equation*}
\int_{K}|D^ku|^2\leq\int_{B_{R/2^k}}|D^ku|^2\leq CR^{2\gamma+n-2k}\to0.
\end{equation*}
Hence, $D^ku\equiv0$ in any compact $K\subset\R^n$. This implies that $u$ is a polynomial of degree at most $k-1$. However, the growth condition implies that its degree can not exceed $\lfloor \gamma\rfloor$.
\end{proof}

\begin{remark}
Theorem \ref{t:liouville} in particular says that:
\begin{itemize}
\item[(i)] if $\gamma<1$, then $u(x)\equiv c$ for some $c\in\R$ (i.e. $u$ is constant);
\item[(ii)] if $\gamma<2$, then $u(x)=a\cdot x+b$ for some $a\in\R^n,b\in\R$ (i.e. $u$ is linear).

\end{itemize}
\end{remark}

\begin{remark}\label{r:superpolynomial}
Without assuming a polynomial bound, there exist entire harmonic functions which are not polynomials. Let $a,b\in\R^n$ with $|a|=|b|$ and $a\cdot b=0$. Then $u(x)=e^{a\cdot x}\sin(b\cdot x)$ is entire harmonic in $\R^n$. In dimension $n=2$ one can easily provide entire harmonic functions with arbitrary high growth by taking the real or the imaginary parts of the holomorphic complex function $e^z$ composed with itself $k$ times (for any $k\in\mathbb N$).
\end{remark}

\begin{remark}
Theorem \ref{t:liouville} holds true even if $u$ is an entire solution of $-\mathrm{div}(A\nabla u)=0$ in $\R^n$ where $A$ is a constant coefficient uniformly elliptic matrix.
\end{remark}

\subsection{$L^\infty$ bounds}
Aim of this section is the proof of the famous $L^2\to L^\infty$ estimate proved by De Giorgi \cite{DeG57} and Nash \cite{Nas57,Nas58}. This is just the first part of their proof of the Hilbert XIXth problem. Then, there is a different proof by Moser \cite{Mos60}. Both De Giorgi's and Moser's approaches implement an iteration procedure involving the Sobolev embedding inequality $H^1_0\subset L^{2^*}$; that is,
\begin{equation}\label{sobolev}
\left(\int_{B_R}|u|^{2^*}\right)^{2/2^*}\leq C\int_{B_R}|\nabla u|^2,\qquad\mathrm{with \ }2^*:=\frac{2n}{n-2}>2, \ (2^* \ \mathrm{is \ any \ }p>0 \ \mathrm{if \ }n=2).
\end{equation}
Let us remark that the constant above depends on the dimension only, while for $n=2$ it depends also on $R$.

\begin{Theorem}\label{t:Linfty}
Let $p>n/2$, $q>n$, $0<r<R\leq1$. Then there exists a constant $C>0$ depending only on $n,p,q,r,R$ and the ellipticity constants in $B_R$ such that
\begin{equation*}\label{Linfty}
\|u\|_{L^\infty(B_{r})}\leq C(\|u\|_{L^2(B_R)}+\|f\|_{L^p(B_R)}+\|F\|_{L^q(B_R)})
\end{equation*}
for any $u\in H^1(B_R)$ weak solution to \eqref{E} in $B_R$.
\end{Theorem}

\begin{remark}
As usual, the constant in the above theorem explodes as $R-r\to0$.
\end{remark}

\begin{proof}
Let us divide the proof into three steps. We refer to \cite{FerRos22,Vas16} for futher details.\\

\noindent{\bf Step 1: Caccioppoli inequality for truncated solutions.}
The first step is to prove a Caccioppoli inequality for the functions $v=(u-b)_+$ and $w=(u-b)_-$, where $b\in\R$ and
$$f_+(x)=\max\{f(x),0\},\qquad f_-(x)=\max\{-f(x),0\}.$$
Let us remark that $v,w\in H^1(B_R)$ since $u\in H^1(B_R)$. Moreover, whenever $v>0$ one has $\nabla v=\nabla u$ and of course $\nabla v=0$ wherever $v=0$. In analogy, whenever $w>0$ one has $\nabla w=-\nabla u$ and of course $\nabla w=0$ wherever $w=0$ (this fact is easy to prove and can be found for instance in \cite[Lemma 7.6]{GilTru77}). Then, we prove the Caccioppoli inequality for $v$ being the case of $w$ very similar. Let us fix two radii $0<r<\rho\leq R$. Let us test the equation \eqref{E} for $u$ in $B_R$ against $\eta^2v\in H^1_0(B_\rho)$ where $\eta$ is as in the proof of Proposition \ref{p:caccioppoli}; that is, $\eta\in C^\infty_c(B_\rho)$ with $0\leq\eta\leq1$, radially decreasing cut-off function with $\eta=1$ in $B_r$, $|\nabla\eta|\leq2(\rho-r)^{-1}$:
\begin{equation*}
\int_{B_\rho\cap\{v>0\}}A\nabla u\cdot\nabla(\eta^2v)=\int_{B_\rho\cap\{v>0\}}f\eta^2v-F\cdot\nabla(\eta^2v).
\end{equation*}
Then, recalling that whenever $v>0$ one has $\nabla v=\nabla u$, and proceeding as in the proof of Proposition \ref{p:caccioppoli}, one obtains
\begin{equation}\label{cacciov}
\int_{B_\rho\cap\{v>0\}}|\nabla(\eta v)|^2\leq C\left(\int_{B_\rho\cap\{v>0\}}v^2|\nabla\eta|^2+\int_{B_\rho\cap\{v>0\}}f\eta^2v+\int_{B_\rho\cap\{v>0\}}|F|^2\right).
\end{equation}
Similarly, testing \eqref{E} for $u$ in $B_R$ against $\eta^2w$ one gets
\begin{equation}\label{cacciow}
\int_{B_\rho\cap\{w>0\}}|\nabla(\eta w)|^2\leq C\left(\int_{B_\rho\cap\{w>0\}}w^2|\nabla\eta|^2+\int_{B_\rho\cap\{w>0\}}f\eta^2w+\int_{B_\rho\cap\{w>0\}}|F|^2\right).
\end{equation}

\noindent{\bf Step 2: no spike lemma.}
This is the main step of the proof. We aim to prove that, provided
\begin{equation*}
\|f\|_{L^p(B_R)}+\|F\|_{L^q(B_R)}\leq 1,
\end{equation*}
there exists $\delta\in(0,1)$ such that
\begin{itemize}
\item[(i)] if 
\begin{equation*}
\int_{B_R}|u_+|^2\leq\delta\qquad \Rightarrow \qquad u\leq 1 \ \mathrm{almost \ everywhere \ in \ } B_r.
\end{equation*}
\item[(ii)] if 
\begin{equation*}
\int_{B_R}|u_-|^2\leq\delta\qquad \Rightarrow \qquad u\geq -1 \ \mathrm{almost \ everywhere \ in \ } B_r.
\end{equation*}
\end{itemize}
We just prove (i) since (ii) is analogous. Let $b_k=1-2^{-k}$ so that $b_0=0$ and $b_k\nearrow b_\infty=1$. Let $r_k=(R-r)2^{-k}+r$ so that $r_0=R$ and $r_k\searrow r_\infty=r$. Let $D_k=B_{r_k}$ with
\begin{equation*}
D_{k+1}\subset B_{\rho_k}\subset D_k\qquad\mathrm{where \ }\rho_k=\frac{r_k+r_{k+1}}{2}.
\end{equation*}
Let us also notice that $r_{k}-r_{k+1}=2^{-(k+1)}(R-r)$. Let us define
$$v_k=(u-b_k)_+\qquad \mathrm{and}\qquad E_k=\int_{D_k}v_k^2.$$
We notice that $v_{k+1}\leq v_k$ and that $E_{k+1}\leq E_k\leq...\leq E_0\leq\delta$ with $\delta\in(0,1)$ to be announced. Let us now introduce for any $k\in\mathbb N$ the radially decreasing cut-off function $\eta_k\in C^\infty_c(B_{\rho_k})$ with $0\leq\eta_k\leq1$ and $\eta_k\equiv1$ in $D_{k+1}$, with $|\nabla\eta_k|\leq2(\rho_k-r_{k+1})^{-1}\leq \frac{c}{R-r}2^{k+1}$. Now we are going to apply \eqref{cacciov} with $v=v_{k+1}$, $\eta=\eta_k$, $r=r_{k+1}$, $\rho=\rho_k$; that is,
\begin{eqnarray*}
\int_{B_{\rho_k}}|\nabla(\eta_k v_{k+1})|^2&\leq& C\left(\int_{B_{\rho_k}}v_{k+1}^2|\nabla\eta_k|^2+\int_{B_{\rho_k}}f\eta^2_kv_{k+1}+\int_{B_{\rho_k}\cap\{v_{k+1}>0\}}|F|^2\right)\\
&=&C(I_k^1+I_k^2+I_k^3).
\end{eqnarray*}
Remember that all the above integrals are actually computed in $B_{\rho_k}\cap\{v_{k+1}>0\}$. In the last integral the latter information has to be written explicitly since the dependence on $v_{k+1}$ is missing.

First, since $v_{k+1}\leq v_k$, $B_{\rho_k}\subset D_k$ and $|\nabla\eta_k|\leq \frac{c}{R-r}2^{k+1}$
$$I_k^1=\int_{B_{\rho_k}}v_{k+1}^2|\nabla\eta_k|^2\leq \frac{c}{(R-r)^2}2^{2(k+1)}E_k.$$
Let us fix $\tau>2$ to be announced. If $n\geq3$ we will take $\tau=2^*$, and if $n=2$ the choice of $\tau>2$ will depend on $p,q$.
Let us take $\alpha>1$ so that
$$\frac{1}{p}+\frac{1}{\tau}+\frac{1}{\alpha}=1.$$
Notice that when $n\geq3$ (and with the choice $\tau=2^*$) the existence of such $\alpha>1$ needs $p>\frac{2n}{n+2}$ which is weaker than $p>\frac{n}{2}$. Instead, when $n=2$ we are requiring $p>1$ which gives the existence of such $\alpha>1$ if we choose $\tau>\frac{p}{p-1}$. 
Then, applying the H\"older inequality with exponents $p,\tau,\alpha$, we have
\begin{eqnarray*}
I_k^2&=&\int_{B_{\rho_k}}f\eta^2_kv_{k+1}\\
&\leq& \|f\|_{L^p(B_{\rho_k})}\Big(\int_{B_{\rho_k}}|\eta_kv_{k+1}|^{\tau}\Big)^{\frac{1}{\tau}}\Big(\int_{B_{\rho_k}}\chi_{\{v_{k+1}>0\}}\Big)^{1-\frac{1}{\tau}-\frac{1}{p}}\\
&\leq&\varepsilon C\int_{B_{\rho_k}}|\nabla(\eta_k v_{k+1})|^2+\frac{C}{\varepsilon}\Big(\int_{B_{\rho_k}}\chi_{\{v_{k+1}>0\}}\Big)^{2-\frac{2}{\tau}-\frac{2}{p}}.
\end{eqnarray*}
In the last inequality we used the Young inequality with a small $\varepsilon>0$ to be announced, the Sobolev inequality \eqref{sobolev} and the assumption $\|f\|_{L^p(B_R)}\leq1$. Then, since
$$\{v_{k+1}>0\}=\{u-b_{k+1}>0\}=\{u-b_k>b_{k+1}-b_k\}=\{v_k>2^{-(k+1)}\}=\{v^2_k>2^{-2(k+1)}\},$$
we have
\begin{equation*}
\int_{B_{\rho_k}}\chi_{\{v_{k+1}>0\}}=\int_{B_{\rho_k}}\chi_{\{v^2_{k}>2^{-2(k+1)}\}}\leq 2^{2(k+1)}E_k.
\end{equation*}
Then there exists a constant $C_1>1$ depending on $p,\tau$ such that
\begin{equation*}
I_k^2\leq \varepsilon C\int_{B_{\rho_k}}|\nabla(\eta_k v_{k+1})|^2+\frac{C_1^{k+1}}{\varepsilon}E_k^{2-\frac{2}{\tau}-\frac{2}{p}}.
\end{equation*}

Then, applying the H\"older inequality with exponents $q/2$ and $\beta>1$ (the existence of such $\beta>1$ needs $q>2$ which is weaker than $q>n$) so that
$$\frac{2}{q}+\frac{1}{\beta}=1,$$
we have
\begin{eqnarray*}
I_k^3&=&\int_{B_{\rho_k}\cap\{v_{k+1}>0\}}|F|^2\\
&\leq& \|F\|^2_{L^q(B_{\rho_k})}\Big(\int_{B_{\rho_k}}\chi_{\{v_{k+1}>0\}}\Big)^{1-\frac{2}{q}}\\
&\leq&C_2^{k+1}E_k^{1-\frac{2}{q}}.
\end{eqnarray*}
where $C_2>1$ depends on $q$ and we used the assumption $\|F\|_{L^q(B_R)}\leq1$. Putting together the information above and choosing $\varepsilon>0$ small enough to reabsorb the gradient term in the left hand side, we finally obtain the existence of a constant $C_0>1$ (depending on $p,q,\tau$) such that
\begin{equation*}
\int_{B_{\rho_k}}|\nabla(\eta_k v_{k+1})|^2\leq \frac{C_0^{k+1}}{(R-r)^2}(E_k+E_k^{2-\frac{2}{\tau}-\frac{2}{p}}+E_k^{1-\frac{2}{q}})
\end{equation*}
From the other side, using the H\"older inequality with exponents $\frac{\tau}{2}$ and $\frac{\tau}{\tau-2}$ (recall that $\tau>2$), and using again the Sobolev inequality \eqref{sobolev}, we get
\begin{eqnarray*}
E_{k+1}&=&\int_{D_{k+1}}v_{k+1}^2\\
&\leq& \Big(\int_{D_{k+1}}|v_{k+1}|^{\tau}\Big)^{\frac{2}{\tau}} \Big(\int_{D_{k+1}}\chi_{\{v_{k+1}>0\}}\Big)^{\frac{\tau-2}{\tau}}\\
&\leq& \Big(\int_{B_{\rho_k}}|\eta_kv_{k+1}|^{\tau}\Big)^{\frac{2}{\tau}} C_3^{k+1}E_k^{\frac{\tau-2}{\tau}}\\
&\leq& \frac{\tilde C^{k+1}}{(R-r)^2}E_k(E_k^{1-\frac{2}{\tau}}+E_k^{2-\frac{4}{\tau}-\frac{2}{p}}+E_k^{1-\frac{2}{\tau}-\frac{2}{q}})\leq \frac{\tilde C^{k+1}}{(R-r)^2}E_k^{1+\gamma}
\end{eqnarray*}
where $C_3,\tilde C>1$ and
\begin{equation*}
\gamma:=\min\left\{1-\frac{2}{\tau},2-\frac{4}{\tau}-\frac{2}{p},1-\frac{2}{\tau}-\frac{2}{q}\right\}>0.
\end{equation*}
Notice that the positivity of $\gamma$ holds true when $n\geq3$ since $\tau=2^*$ and since $p>n/2$ and $q>n$. Moreover, if $n=2$ we know that $p>1$ and $q>2$ so that the latter is true by choosing $\tau>\frac{2p}{p-1}$ and $\tau>\frac{2q}{q-2}$. Notice that, in estimating with the smallest exponent, we also use the fact that $E_k\leq\delta<1$. Then, iterating the inequality
\begin{equation*}
\begin{cases}
E_{k+1}\leq \frac{\tilde C^{k+1}}{(R-r)^2}E_{k}^{1+\gamma}\\
E_0\leq\delta,
\end{cases}
\end{equation*}
we get
\begin{equation*}
E_k\leq \frac{\tilde C^{\sum_{i=0}^ki(1+\gamma)^{k-i}}}{(R-r)^{2\sum_{i=0}^k(1+\gamma)^{k-i}}}E_0^{(1+\gamma)^k}\leq\left(\frac{\tilde C^{\sum_{i=0}^{k}\frac{i}{(1+\gamma)^i}}}{(R-r)^{2\sum_{i=0}^{k}\frac{1}{(1+\gamma)^i}}}\delta\right)^{(1+\gamma)^k}.
\end{equation*}
So, since $\sum_{i=0}^{+\infty}\frac{i}{(1+\gamma)^i}$ and $\sum_{i=0}^{+\infty}\frac{1}{(1+\gamma)^i}$ are convergent and given $S_1,S_2$ their sums, we can find a small $\delta\in(0,1)$ such that
$$\frac{\tilde C^{S_1}}{(R-r)^{2S_2}}\delta<1,$$
so that $E_k\to0$ as $k\to\infty$. By dominated convergence theorem this implies that
$$
\lim_{k\to\infty}\int_{B_R}(u-b_k)^2_+\chi_{D_k}=\int_{B_r}(u-1)_+^2=0;
$$
that is $u\leq1$ almost everywhere in $B_r$.\\

\noindent{\bf Step 3: normalization.}
Let us define
$$v=\theta u,\qquad\mathrm{with}\qquad \theta=\frac{\sqrt\delta}{\|u\|_{L^2(B_R)}+\|f\|_{L^p(B_R)}+\|F\|_{L^q(B_R)}},$$
and with $\delta\in(0,1)$ such that {\bf Step 2} works on $v_+,v_-$. Then $|v|\leq1$ in $B_r$ so that
$$\|u\|_{L^\infty(B_r)}\leq\frac{1}{\sqrt\delta}(\|u\|_{L^2(B_R)}+\|f\|_{L^p(B_R)}+\|F\|_{L^q(B_R)}).$$

\end{proof}

\section{A priori $C^{0,\alpha}$ estimates, a priori $C^{1,\alpha}$ estimates}\label{s:5}
The aim of this section is to provide a priori regularity estimates for solutions by a scaling argument which involves a blow-up procedure and the use of the Liouville theorem in the previous section. This procedure requires the coefficients to be at least continuous in order to end up with a constant coefficient matrix after blow-up and make use of the polynomial Liouville Theorem \ref{t:liouville}. We would like to stress that the De Giorgi-Nash-Moser theorem (which is not treated in these notes) proves local $\alpha$-H\"older continuity of solutions just requiring bounded measurable coefficients. However, even in the case of a zero right hand side, the exponent $\alpha$ is not allowed to be any real number in $(0,1)$, but has an implicit upper bound which depends on the ellipticity ratio $\lambda/\Lambda$, and this is optimal for general bounded measurable coefficients. As we will see, assuming additionally the continuity of coefficients, this threshold is removed.

\subsection{A priori $C^{0,\alpha}$ estimates}
In the next result we are going to assume continuity of the coefficients in $\overline{B_R}$ for a given $R>0$; that is, uniform continuity. So we can assume that there exists a modulus of continuity $\omega$ such that
\begin{equation*}
\|A\|_{C^{0,\omega(\cdot)}(B_R)}=\|A\|_{L^{\infty}(B_R)}+\sup_{\substack{x,y\in B_R \\ x\neq y}}\frac{|A(x)-A(y)|}{\omega(|x-y|)}<\infty.
\end{equation*}
The latter expression has to be intended for any component $a_{ij}$.

\begin{Theorem}[A priori $C^{0,\alpha}$ estimates]\label{t:C0alphaPriori}
Let $p>n/2$, $q>n$, $0<r<R$. Let $\alpha\in(0,1)$ such that
\begin{equation*}
\alpha\leq\min\{2-n/p,1-n/q\}.
\end{equation*}
Let $A\in C^{0,\omega(\cdot)}(B_R)$ with $\|A\|_{C^{0,\omega(\cdot)}(B_R)}\leq\overline L$ and $\omega$ is any given modulus of continuity. Then, there exists a constant $C>0$ depending only on $n,p,q,\alpha,r,R$, the ellipticity constants in $B_R$ and $\overline L$ such that
\begin{equation}\label{C0alphaE}
\|u\|_{C^{0,\alpha}(B_{r})}\leq C(\|u\|_{L^2(B_R)}+\|f\|_{L^p(B_R)}+\|F\|_{L^q(B_R)})
\end{equation}
for any weak solution of \eqref{E} in $B_R$ which belongs to $C^{0,\alpha}_\loc(B_R)$.
\end{Theorem}
As usual, the constant in the above estimate explodes as $R-r\to0$.

The proof we propose is based on a contradiction argument which involves scalings and blow-ups, in the spirit of Simon's approach in \cite{Sim97}. The idea is that, having continuous coefficients, the blow-up procedure is a zooming around points which leave in a compact set, and so up to select a subsequence one can compare the behaviour of solutions to variable coefficient PDEs with solutions with constant coefficients, which are regular and quite rigid.

\begin{proof}[Proof of Theorem \ref{t:C0alphaPriori}]
Wlog we take $r=1/2$ and $R=1$. We divide the proof into four steps.\\

\noindent{\bf Step 1: the contradiction argument.}
Let us suppose by contradiction the existence of sequences of data $A_k,f_k,F_k$ and of associated solutions $u_k$ such that
\begin{equation*}
\|u_k\|_{C^{0,\alpha}(B_{1/2})}> k(\|u_k\|_{L^2(B_1)}+\|f_k\|_{L^p(B_1)}+\|F_k\|_{L^q(B_1)}):=kI_k,
\end{equation*}
with $p,q,\alpha$ as in the statement and $A_k$ have the same uniform ellipticity constants $\lambda,\Lambda,L$ in $B_1$ and the same common modulus of continuity $\omega$; that is, they are equibounded and equicontinuous with
\begin{equation*}
\|A_k\|_{C^{0,\omega(\cdot)}(B_1)}=\|A_k\|_{L^{\infty}(B_1)}+\sup_{\substack{x,y\in B_1 \\ x\neq y}}\frac{|A_k(x)-A_k(y)|}{\omega(|x-y|)}\leq \overline L.
\end{equation*}
The latter expression has to be intended for any component $a_{ij}^k$. 

Let us remark that, from now on, we may pass to subsequences multiple times within this proof. This is not restrictive, as long as a contradiction is reached along at least one particular subsequence.

By Theorem \ref{t:Linfty} we know that for any $0<r<1$ there exists a constant $c(r)>0$ (remember that $c(r)$ may explode as $r\to1^-$) such that
\begin{equation*}
\|u_k\|_{L^\infty(B_r)}\leq c(r)I_k.
\end{equation*}
Then,
\begin{equation*}
\|u_k\|_{C^{0,\alpha}(B_{1/2})}=\|u_k\|_{L^\infty(B_{1/2})}+[u_k]_{C^{0,\alpha}(B_{1/2})}\leq c(1/2)I_k+[u_k]_{C^{0,\alpha}(B_{1/2})}.
\end{equation*}
Let us consider a radially decreasing cut-off function $\eta\in C^\infty_c(B_{3/4})$ with $\eta\equiv1$ in $B_{1/2}$ and $0\leq\eta\leq1$. Then
\begin{equation*}
M_k:=[\eta u_k]_{C^{0,\alpha}(B_{1})}\geq [u_k]_{C^{0,\alpha}(B_{1/2})}\geq (k-c(1/2))I_k\geq \frac{k}{2}I_k,
\end{equation*}
for $k$ big enough. This in particular gives that $I_k\leq 2k^{-1}M_k$. By definition of supremum, there exist two sequences of points $x_k,y_k\in B_1$ (blow-up points) such that $x_k\neq y_k$ and
\begin{equation*}
\frac{|\eta u_k(x_k)-\eta u_k(y_k)|}{|x_k-y_k|^\alpha}\geq \frac{M_k}{2}.
\end{equation*}
Since $\overline{\mathrm{supp}\eta}\subset B_s\subset B_{3/4}$ (for some $s\in(1/2,3/4)$) we can assume up to relabeling that $x_k\in B_s\subset B_{3/4}$. Then, actually also $y_k$ belongs to $B_{3/4}$. In fact, if $y_k\in B_{3/4}^c$, by taking the segment connecting $x_k$ and $y_k$ one could take the point $\hat y_k$ on the segment which lies on $\partial B_s$. We observe that $|x_k-\hat y_k|<|x_k-y_k|$ and $|\eta u_k(x_k)-\eta u_k(\hat y_k)|=|\eta u_k(x_k)-\eta u_k(y_k)|=|\eta u_k(x_k)|$, so one could take $\hat y_k$ in place of $y_k$. Let us call $r_k=|x_k-y_k|$ the distance between the blow-up points. We have
\begin{equation*}
\|u_k\|_{L^\infty(B_{3/4})}\leq c(3/4)I_k\leq 2c(3/4)\frac{M_k}{k}\leq \frac{4c(3/4)}{k}\frac{|\eta u_k(x_k)-\eta u_k(y_k)|}{|x_k-y_k|^\alpha}\leq \frac{8c(3/4)}{kr_k^{\alpha}}\|u_k\|_{L^\infty(B_{3/4})}.
\end{equation*}
Hence $r_k\leq ck^{-1/\alpha}\to0$, i.e. the blow-up points are collapsing in the limit.\\

\noindent{\bf Step 2: the blow-up sequences.} Let us define two blow-up sequences

\begin{equation*}
v_k(x)=\frac{\eta u_k(x_k+r_kx)-\eta u_k(x_k)}{M_kr_k^\alpha},\qquad w_k(x)=\eta(x_k)\frac{u_k(x_k+r_kx)- u_k(x_k)}{M_kr_k^\alpha}.
\end{equation*}
They are well defined as long as $x_k+r_kx\in B_1$; that is, $$x\in\Omega_k:=\frac{B_1-x_k}{r_k},$$
which are called blow-up domains. From one side, we want to prove that the $v_k$s enjoy some equi-H\"older continuity which gives compactness and some regularity and growth properties of the limit. From the other side, we want to show that the $w_k$s have the same asymptotic behaviour (i.e. same limit of the $v_k$s) and they solve some rescaled equations which bring an equation to the limit too. Since $x_k\in B_{3/4}$, the blow-up domains are exhausting the whole of $\R^n$; that is
$$\Omega_\infty=\{x\in\R^n \ \mathrm{such \  that \ there \ exists \ } \overline k \ \mathrm{such \  that \ } x\in\Omega_k, \ \forall k\geq\overline k  \}=\R^n.$$
In fact, given $x_0\in\R^n$
\begin{equation*}
|x_k+r_kx_0|\leq |x_k|+r_k|x_0|<\frac{3}{4}+r_k|x_0|<1
\end{equation*}
for any $k\geq\overline k$ with $\overline k$ depending on $|x_0|$ since $r_k\to0$. Now we derive some properties of the blow-up sequences. Let $x,y\in\Omega_k$, then
\begin{equation}\label{equiholder}
|v_k(x)-v_k(y)|=\frac{|\eta u_k(x_k+r_kx)-\eta u_k(x_k+r_ky)|}{M_kr_k^\alpha}\leq|x-y|^\alpha.
\end{equation}
Hence $[v_k]_{C^{0,\alpha}(B_R)}\leq1$ for any $R>0$ (since $B_R$ is definitively contained in any $\Omega_k$). Moreover since $v_k(0)=0$, we have
$$\|v_k\|_{L^\infty(B_R)}=\sup_{x\in B_R}|v_k(x)-v_k(0)|\leq |x|^\alpha\leq R^\alpha.$$
Given the compact set $\overline{B_R}\subset\R^n$, the sequence $v_k$ is equibounded and equicontinuous on $\overline{B_R}$ (actually equi $\alpha$-H\"older continuous), then by the Ascoli-Arzel\'a theorem it converges (up to pass to a subsequence) uniformly in $\overline{B_R}$ to some limiting profile $v$. The convergence is in $C^{0,\beta}(B_R)$ for any $0<\beta<\alpha$. By an exhaustion of $\R^n$ with countably many compact sets $\overline{B_R}$, and a diagonal argument along subsequences, one can select a unique limiting profile $v$ defined in the whole of $\R^n$. Moreover we observe that
\begin{equation*}
|v_k(0)-v_k(\xi_k)|\geq \frac{1}{2},\qquad\mathrm{with \ }\xi_k=\frac{y_k-x_k}{r_k}\in \mathbb S^{n-1}
\end{equation*}
and up to subsequences $\xi_k\to\xi\in \mathbb S^{n-1}$. Hence, by uniform convergence on compact sets we have
\begin{equation*}
|v(0)-v(\xi)|\geq \frac{1}{2};
\end{equation*}
that is, $v$ is not constant. Moreover for any $x\in\R^n$, again by the uniform convergence, the fact that $v(0)=0$ and the equi-H\"older continuity of the $v_k$s \eqref{equiholder}, one has that $v$ has sublinear growth
\begin{equation}\label{sublinear}
|v(x)|\leq|x|^\alpha.
\end{equation}
The two sequences converge to the same limit uniformly on compact sets, since taken $x\in \overline{B_R}\subset\R^n$
\begin{eqnarray*}
|v_k(x)-w_k(x)|&=&\frac{|\eta(x_k+r_kx)-\eta(x_k)|\cdot|u_k(x_k+r_kx)|}{M_kr_k^\alpha}\\
&\leq&\frac{\ell r_k|x|\cdot \|u_k\|_{L^\infty(B_{4/5})}}{M_kr_k^\alpha}\leq c(4/5)Rr_k^{1-\alpha}k^{-1}\to0.
\end{eqnarray*}
This says also that $w_k$ converges to the same limit $v$ uniformly on any $\overline{B_R}$.\\

\noindent{\bf Step 3: the limiting profile is entire "harmonic".}
Along a subsequence, $x_k,y_k\to \overline x$. Moreover, the equicontinuity and equiboundedness of $A_k$s give their uniform convergence (up to subsequences) on compact sets of $\R^n$ to a constant coefficient uniformly elliptic matrix; that is,
\begin{equation*}
A_k(x_k+r_kx)\to \overline A=\overline A(\overline x).
\end{equation*}
Then, given a test function $\phi\in C^\infty_c(\R^n)$, its support will be contained in a possibly large ball; that is, $\mathrm{supp}\phi\subset B_R\subset\Omega_k$ definitively. Then
\begin{eqnarray}\label{eqwk}
\int_{B_R}A_k(x_k+r_kx)\nabla w_k(x)\cdot\nabla\phi(x)&=&\frac{r_k^{2-\alpha}\eta(x_k)}{M_k}\int_{B_R}f_k(x_k+r_kx)\phi(x)\nonumber\\
&&-\frac{r_k^{1-\alpha}\eta(x_k)}{M_k}\int_{B_R}F_k(x_k+r_kx)\cdot\nabla\phi(x)=T^1_k+T^2_k.
\end{eqnarray}
First
\begin{eqnarray*}
\left|\int_{B_R}f_k(x_k+r_kx)\phi(x)dx\right|&\leq& \|\phi\|_{L^\infty(\mathrm{supp}\phi)}\int_{B_R}|f_k(x_k+r_kx)|dx\\
&=&\|\phi\|_{L^\infty(\mathrm{supp}\phi)}\int_{B_{r_kR}(x_k)}|f_k(y)|r_k^{-n}dy.
\end{eqnarray*}
Here $y=x_k+r_kx\in B_{r_kR}(x_k)\subset B_1$. Hence, remembering that $I_k\leq 2k^{-1}M_k$,
\begin{equation*}
|T^1_k|\leq\frac{r_k^{2-\alpha-n}}{M_k}\|\phi\|_{L^\infty(\mathrm{supp}\phi)}\|f_k\|_{L^p(B_1)}|B_{r_kR}(x_k)|^{1-\frac{1}{p}}\leq c\|\phi\|_{L^\infty(\mathrm{supp}\phi)}r_k^{2-\frac{n}{p}-\alpha}k^{-1}\to0
\end{equation*}
since $\alpha\leq 2-n/p$. Similarly
\begin{eqnarray*}
\left|\int_{B_R}F_k(x_k+r_kx)\cdot\nabla\phi(x)dx\right|&\leq& \|\nabla\phi\|_{L^2(\mathrm{supp}\phi)}\left(\int_{B_R}|F_k(x_k+r_kx)|^2dx\right)^{1/2}\\
&=&\|\nabla\phi\|_{L^2(\mathrm{supp}\phi)}\left(\int_{B_{r_kR}(x_k)}|F_k(y)|^2r_k^{-n}dy\right)^{1/2}\\
&\leq&c\|\nabla\phi\|_{L^2(\mathrm{supp}\phi)}\|F_k\|_{L^q(B_1)}r_k^{-\frac{n}{q}}.
\end{eqnarray*}
Hence
\begin{equation*}
|T^2_k|\leq c\|\nabla\phi\|_{L^2(\mathrm{supp}\phi)}r_k^{1-\frac{n}{q}-\alpha}k^{-1}\to0
\end{equation*}
since $\alpha\leq 1-n/q$. This in particular means that there exists a sequence $\delta_k\to0$ such that
\begin{equation*}
\int_{\mathrm{supp}\phi}A_k(x_k+r_kx)\nabla w_k(x)\cdot\nabla\phi(x)\leq \delta_k(\|\phi\|_{L^\infty(\mathrm{supp}\phi)}+\|\nabla\phi\|_{L^2(\mathrm{supp}\phi)}).
\end{equation*}
Let us choose as test function $\phi=\eta^2w_k$ where $\eta\in C^\infty_c(B_{2R})$ (for a given $R>0$) is a radially decreasing cut-off function with $0\leq\eta\leq1$, $\eta\equiv1$ in $B_R$ and $|\nabla\eta|\leq 2R^{-1}$. Notice that $w_k\in H^1(\Omega_k)$ and hence $\eta^2w_k\in H^1_0(B_{2R})$ (since $B_{2R}\subset\Omega_k$). Hence, by similar computations as in the proof of the Caccioppoli inequality (see Proposition \ref{p:caccioppoli}), one ends up with ($\varepsilon>0$ to be announced)

\begin{eqnarray*}
\int_{B_{2R}}A_k(x_k+r_kx)\nabla(\eta w_k)\cdot\nabla(\eta w_k)&\leq& C_0\varepsilon\int_{B_{2R}}|\nabla(\eta w_k)|^2+\frac{C_0}{\varepsilon}\int_{B_{2R}}w_k^2|\nabla\eta|^2+\Lambda\int_{B_{2R}}w_k^2|\nabla\eta|^2\\
&&+\delta_k\|\eta^2w_k\|_{L^\infty(B_{2R})}+\delta_k\left(\int_{B_{2R}}|\nabla(\eta^2w_k)|^2\right)^{1/2}\\
&\leq&C_0\varepsilon\int_{B_{2R}}|\nabla(\eta w_k)|^2+\frac{1}{R^2}\left(\frac{C_0}{\varepsilon}+\Lambda\right)\int_{B_{2R}}w_k^2\\
&&+\delta_k\|w_k\|_{L^\infty(B_{2R})}+\sqrt 2\delta_k\left(\int_{B_{2R}}|\nabla(\eta w_k)|^2+\frac{1}{R^2}\int_{B_{2R}}w_k^2\right)^{1/2}\\
&\leq&C(\varepsilon+\delta_k)\int_{B_{2R}}|\nabla(\eta w_k)|^2+\frac{C}{\varepsilon}\|w_k\|^2_{L^\infty(B_{2R})}+C.
\end{eqnarray*}
Then choosing $\varepsilon>0$ small enough, there exists a uniform in $k$ contant $c>0$ such that
\begin{equation}\label{boundwk}
\int_{B_{R}}|\nabla w_k|^2\leq C\|w_k\|^2_{L^\infty(B_{2R})}+C\leq c,
\end{equation}
since $w_k$s are uniformly converging on compact sets.
Then, up to further subsequences, the convergence $w_k\to v$ is also weak in $H^1_\loc(\R^n)$ (i.e. on any compact of $\R^n$). Thus, $v\in H^1_\loc(\R^n)$. Now we are going to prove that actually $v$ is entire $\overline A$-harmonic in $\R^n$; that is, for any $\phi\in C^\infty_c(\R^n)$
\begin{equation*}
\int_{\R^n}\overline A\nabla v\cdot\nabla\phi=0.
\end{equation*}
However, going back to \eqref{eqwk}, we already know that the right hand side is vanishing, se we just need to prove that
\begin{equation*}
\int_{\R^n}A_k(x_k+r_kx)\nabla w_k\cdot\nabla\phi\to\int_{\R^n}\overline A\nabla v\cdot\nabla\phi.
\end{equation*}
Then
\begin{equation*}
\int_{\R^n}A_k(x_k+r_kx)\nabla w_k\cdot\nabla\phi=\int_{\R^n}(A_k(x_k+r_kx)-\overline A)\nabla w_k\cdot\nabla\phi+\int_{\R^n}\overline A\nabla w_k\cdot\nabla\phi.
\end{equation*}
The first term in the right hand side is vanishing using the uniform boundedness \eqref{boundwk} together with the equicontinuity of $A_k$s (since $|x_k+r_kx-\overline x|<<1$ uniformly on $\mathrm{supp}\phi$). Then the second term of the right hand side converges to the desired one by weak convergence in $H^1_\loc$.\\

\noindent{\bf Step 4: the conclusion by the Liouville theorem.}
Summing up all the information obtained, we have a limiting profile $v$ which is $H^1_\loc(\R^n)$, and entire $\overline A$-harmonic. Moreover $v$ is not constant since $|v(0)-v(\xi)|\geq 1/2$ and its growth is sublinear by \eqref{sublinear}. Then, this is in contradiction with the Liouville Theorem \ref{t:liouville}.
\end{proof}

\subsection{A priori $C^{1,\alpha}$ estimates}
We would like to acknowledge \cite{SoaTer19} for the proof of the result below, in the spirit of Simon's proof in \cite{Sim97}.

\begin{Theorem}[A priori $C^{1,\alpha}$ estimates]\label{t:C1alphaPriori}
Let $p>n$, $0<r<R$. Let $\alpha\in(0,1)$ such that
\begin{equation*}
\alpha\leq 1-n/p.
\end{equation*}
Let $A\in C^{0,\alpha}(B_R)$ with $\|A\|_{C^{0,\alpha}(B_R)}\leq\overline L$.
Then, there exists a constant $C>0$ depending only on $n,p,\alpha,r,R$, the ellipticity constants in $B_R$ and $\overline L$ such that
\begin{equation}\label{C1alphaE}
\|u\|_{C^{1,\alpha}(B_{r})}\leq C(\|u\|_{L^2(B_R)}+\|f\|_{L^p(B_R)}+\|F\|_{C^{0,\alpha}(B_R)})
\end{equation}
for any weak solution of \eqref{E} in $B_R$ which belongs to $C^{1,\alpha}_\loc(B_R)$.
\end{Theorem}

\begin{proof}[Proof of Theorem \ref{t:C1alphaPriori}]
Wlog we take $r=1/100$ and $R=1$. By reasoning as in Remark \ref{r:scaling-covering} this will easily give an estimate on bigger balls $B_r$ with $1/100<r<1$.\\

\noindent{\bf Step 1: the contradiction argument.}
Let us suppose by contradiction the existence of sequences of data $A_k,f_k,F_k$ and of associated solutions $u_k$ such that
\begin{equation*}
\|u_k\|_{C^{1,\alpha}(B_{r})}> k(\|u_k\|_{L^2(B_1)}+\|f_k\|_{L^p(B_1)}+\|F_k\|_{C^{0,\alpha}(B_1)}):=kI_k,
\end{equation*}
with $p,\alpha$ as in the statement and $A_k$ have the same uniform ellipticity constants $\lambda,\Lambda,L$ in $B_1$ and the same common bound
\begin{equation*}
\|A_k\|_{C^{0,\alpha}(B_1)}\leq \overline L.
\end{equation*}
By Theorem \ref{t:C0alphaPriori} we know that for any $\beta\in(0,1)$ and any $s\in(0,1)$ there exists a constant $c(s,\beta)>0$ (possibly exploding as $s\to1^-$ or $\beta\to1^-$) such that
\begin{equation}\label{unifCbeta}
\|u_k\|_{C^{0,\beta}(B_s)}\leq c(s,\beta)I_k.
\end{equation}
Notice that $p>n$ implies that $2-n/p>1$. The latter estimate in particular comprehends the $L^\infty$ bound given by Theorem \ref{t:Linfty}. Notice also that by definition of supremum, there exists a sequence of points $\zeta_k\in \overline{B_r}$ such that
\begin{equation*}
|\nabla u_k(\zeta_k)|\geq\frac{1}{2}\|\nabla u_k\|_{L^\infty(B_r)}.
\end{equation*}
Hence
\begin{eqnarray*}
|\nabla u_k(\zeta_k)|^2&=&\frac{1}{|B_r|}\int_{B_r}|\nabla u_k(\zeta_k)|^2\\
&\leq& c\left(\int_{B_r}|\nabla u_k(\zeta_k)-\nabla u_k(x)|^2+\int_{B_r}|\nabla u_k|^2\right)\\
&\leq& c\left([\nabla u_k]_{C^{0,\alpha}(B_r)}^2+I_k^2\right),
\end{eqnarray*}
which implies
\begin{equation*}
\|\nabla u_k\|_{L^\infty(B_r)}\leq c([\nabla u_k]_{C^{0,\alpha}(B_r)}+I_k).
\end{equation*}
Hence, the previous inequality together with the $L^\infty$ bound of the $u_k$s give the existence of a positive constant such that
\begin{equation*}
\|u_k\|_{C^{1,\alpha}(B_r)}\leq c([\nabla u_k]_{C^{0,\alpha}(B_r)}+I_k);
\end{equation*}
that is, for $k$ big enough
\begin{equation*}
[\nabla u_k]_{C^{0,\alpha}(B_r)}\geq ckI_k.
\end{equation*}
Let us consider a radially decreasing cut-off function $\eta\in C^\infty_c(B_{2r})$ with $\eta\equiv1$ in $B_{r}$ and $0\leq\eta\leq1$. Then
\begin{equation*}
M_k:=[\nabla(\eta u_k)]_{C^{0,\alpha}(B_{1})}\geq [\nabla u_k]_{C^{0,\alpha}(B_{r})}\geq ckI_k.
\end{equation*}
Let $x_k,y_k$ be two sequences of points such that
\begin{equation*}
\frac{|\nabla(\eta u_k)(x_k)-\nabla(\eta u_k)(y_k)|}{|x_k-y_k|^\alpha}\geq\frac{1}{2}M_k.
\end{equation*}
Let $r_k=|x_k-y_k|$. Reasoning as in the proof of Theorem \ref{t:C0alphaPriori}, we can assume wlog that $x_k,y_k\in B_{2r}$. At this point of the proof, we can not say that $r_k\to0$ but we just know that $0<r_k\leq 4r$.\\

\noindent{\bf Step 2: the blow-up sequences.} Let us define two blow-up sequences

\begin{equation*}
v_k(x)=\frac{\eta(x_k+r_kx)(u_k(x_k+r_kx)- u_k(x_k))-\eta(x_k)\nabla u_k(x_k)\cdot r_kx}{M_kr_k^{1+\alpha}},
\end{equation*}
\begin{equation*}
 w_k(x)=\frac{\eta(x_k)(u_k(x_k+r_kx)- u_k(x_k))-\eta(x_k)\nabla u_k(x_k)\cdot r_kx}{M_kr_k^{1+\alpha}}.
\end{equation*}
They are well defined as long as
$$x\in\Omega_k:=\frac{B_1-x_k}{r_k}.$$
Here, we can not infer directly that $\Omega_\infty=\R^n$ since we lack the information $r_k\to0$. The latter will be true after some more reasonings. First, by the choice $r=1/100$ we can infer that $\overline{B_2}\subset\Omega_k$ definitively for big $k$. Hence, we state some properties of the sequence $v_k$ on $\overline{B_R}$ for any $R>0$ such that $\overline{B_R}$ is contained in $\Omega_\infty$. Notice that at least for $0<R\leq2$ this is the case. Given $x,y\in \overline{B_R}\subset\Omega_k$ for $k$ big enough, then
\begin{eqnarray*}
|\nabla v_k(x)-\nabla v_k(y)|&=&\frac{1}{M_kr_k^\alpha}|\nabla (\eta u_k)(x_k+r_kx)-\nabla (\eta u_k)(x_k+r_ky)|\\
&&+\frac{|u_k(x_k)|}{M_kr_k^\alpha}|\nabla\eta(x_k+r_kx)-\nabla\eta(x_k+r_ky)|\\
&\leq&|x-y|^\alpha+\|u_k\|_{L^\infty(B_{2r})}\frac{r_k^{1-\alpha}}{M_k}\ell|x-y|\\
&\leq&|x-y|^\alpha+ck^{-1}(4r)^{1-\alpha}\ell R\leq 2|x-y|^\alpha.
\end{eqnarray*}
In the previous lines we used the Lipschitz continuity of $\nabla\eta$, the $L^\infty$ bound of the $u_k$s and we took $k$ big enough. In other words, fixed $R>0$ such that $\overline{B_R}\subset\Omega_k$ for $k\geq \overline k$, up to enlarge $\overline k$
\begin{equation*}
[\nabla v_k]_{C^{0,\alpha}(B_R)}\leq2.
\end{equation*}
This, together with $\nabla v_k(0)=0$ gives for $x\in \overline{B_R}$
$$|\nabla v_k(x)|\leq 2|x|^\alpha\leq 2R^\alpha;$$
that is
\begin{equation*}
\|\nabla v_k\|_{L^\infty(B_R)}\leq 2R^\alpha.
\end{equation*}
Moreover since $\overline{B_R}$ is convex, one can take $\gamma(t)=tx$ which is the segment connecting $0$ to $x\in \overline{B_R}$ for $t\in[0,1]$. Then, since $v_k(0)=0$
\begin{eqnarray*}
|v_k(x)|=\left|\int_0^1\nabla v_k(\gamma(t))\cdot\gamma'(t)dt\right|\leq R\|\nabla v_k\|_{L^\infty(B_R)}\leq 2R^{1+\alpha}.
\end{eqnarray*}
This gives a uniform bound for the $v_k$s in $C^{1,\alpha}(B_R)$, and by the Ascoli-Arzel\'a theorem we can infer $C^{1,\beta}(B_R)$ convergence for any $0<\beta<\alpha$, up to pass to a subsequence, to a limiting profile $v$. Notice that $0$ and the sequence $\xi_k=(y_k-x_k)/r_k\in\mathbb S^{n-1}$ both belong to the compact $\overline{B_R}$ for any $R\geq1$. In particular this is true for $\overline{B_2}$. Then, since
\begin{eqnarray*}
|\nabla v_k(0)-\nabla v_k(\xi_k)|&\geq&\frac{|\nabla(\eta u_k)(x_k)-\nabla(\eta u_k)(y_k)|}{M_kr_k^\alpha}-\frac{|u_k(x_k)|\cdot |\nabla\eta(x_k)-\nabla\eta(y_k)|}{M_kr_k^\alpha}\\
&\geq&\frac{1}{2}-c\ell r_k^{1-\alpha}k^{-1}>\frac{1}{4},
\end{eqnarray*}
we have
$$|\nabla v(0)-\nabla v(\xi)|>\frac{1}{4}$$
where $\xi_k\to\xi\in \mathbb S^{n-1}$; that is, $\nabla v$ is not constant (and consequently also $v$ is not constant) in $\overline{B_R}$ if $R\geq1$.\\

\noindent{\bf Step 3: the blow-up points are collapsing.}
We show that $r_k\to\overline r>0$ (along a subsequence) is not possible. In fact, if this is the case then
\begin{eqnarray*}
\left|v_k(x)+\frac{\eta(x_k)\nabla u_k(x_k)\cdot r_kx}{M_kr_k^{1+\alpha}}\right|&=&\frac{\eta(x_k+r_kx)| u_k(x_k+r_kx)-u_k(x_k)|}{M_kr_k^{1+\alpha}}\\
&\leq&\frac{2\|u_k\|_{L^\infty(B_{2r})}}{M_kr_k^{1+\alpha}}\leq ck^{-1}\to0.
\end{eqnarray*}
In the previous estimates we used the fact that $x_k+r_kx\in B_{2r}$ since otherwise $\eta(x_k+r_kx)=0$. Observe that by the definition of $M_k$ and since $\eta u_k\equiv0$ in $B_1\setminus B_{2r}$, then
\begin{equation*}
\|\nabla(\eta u_k)\|_{L^\infty(B_1)}\leq M_k,
\end{equation*}
and hence
\begin{equation}\label{boundgradient}
|\eta(x_k)\nabla u_k(x_k)|=|\nabla(\eta u_k)(x_k)-u_k(x_k)\nabla\eta(x_k)|\leq M_k+c\frac{M_k}{k}\leq 2M_k;
\end{equation}
that is,
\begin{equation*}
\left|\frac{\eta(x_k)\nabla u_k(x_k)}{M_kr_k^{\alpha}}\right|\leq c,\qquad \frac{\eta(x_k)\nabla u_k(x_k)}{M_kr_k^{\alpha}}\to b\in\R^n.
\end{equation*}
The latter convergence holds up to consider a further subsequence. This would give $v=-b\cdot x$ in $\overline{B_2}$, which is in contradiction with the fact that $\nabla v$ is not constant in $\overline{B_2}$. Hence we can conclude that $r_k\to0$. This information gives $\Omega_\infty=\R^n$ and that $B_R\subset\Omega_k$ definitively for any fixed $R>0$. Hence, summarizing the information previously obtained we have a limiting profile $v\in C^{1,\alpha}_\loc(\R^n)$, non constant with non constant gradient, having $v(0)=0$, $\nabla v(0)=0$ and with the subquadratic growth condition for any $x\in\R^n$
\begin{equation}\label{subquadratic}
|v(x)|=\left|\int_0^1\nabla v(tx)\cdot x \, dt\right|\leq \int_0^1|\nabla v(tx)-\nabla v(0)|\cdot |x| \, dt\leq \frac{2}{1+\alpha}|x|^{1+\alpha}.
\end{equation}

The two sequences converge to the same limit uniformly on compact sets. Let us fix $\beta\geq\alpha$ and $s\in(2r,1)$, recall \eqref{unifCbeta} and take $x\in \overline{B_R}\subset\R^n$. Then for $k$ large enough $x_k+r_kx\in B_s$ and
\begin{eqnarray*}
|v_k(x)-w_k(x)|&=&\frac{|\eta(x_k+r_kx)-\eta(x_k)|\cdot|u_k(x_k+r_kx)-u_k(x_k)|}{M_kr_k^{1+\alpha}}\\
&\leq&\frac{\ell c(s,\beta)r_k^{1+\beta}|x|^{1+\beta}}{kr_k^{1+\alpha}}\leq cR^{1+\beta}r_k^{\beta-\alpha}k^{-1}\to0.
\end{eqnarray*}
This says also that $w_k$ converges to the same limit $v$ uniformly on any $\overline{B_R}$.\\

\noindent{\bf Step 4: the limiting profile is entire "harmonic".}
Along a subsequence, $x_k,y_k\to \overline x$. Moreover, the equi-H\"older continuity of $A_k$s give their uniform convergence (up to subsequences) on compact sets of $\R^n$ to a constant coefficient uniformly elliptic matrix; that is,
\begin{equation*}
A_k(x_k+r_kx)\to \overline A=\overline A(\overline x).
\end{equation*}
Then, given a test function $\phi\in C^\infty_c(\R^n)$, its support will be contained in a possibly large ball; that is, $\mathrm{supp}\phi\subset B_R\subset\Omega_k$ definitively. Then
\begin{eqnarray*}\label{eqwk2}
\int_{B_R}A_k(x_k+r_kx)\nabla w_k(x)\cdot\nabla\phi(x)&=&\frac{r_k^{1-\alpha}\eta(x_k)}{M_k}\int_{B_R}f_k(x_k+r_kx)\phi(x)\nonumber\\
&&-\frac{r_k^{-\alpha}\eta(x_k)}{M_k}\int_{B_R}F_k(x_k+r_kx)\cdot\nabla\phi(x)\\
&&-\frac{r_k^{-\alpha}\eta(x_k)}{M_k}\int_{B_R}A_k(x_k+r_kx)\nabla u_k(x_k)\cdot\nabla\phi(x)\\
&=&T^1_k+T^2_k+T^3_k.
\end{eqnarray*}
Working as in the proof of Theorem \ref{t:C0alphaPriori},
\begin{equation*}
|T^1_k|\leq\frac{r_k^{1-\alpha-n}}{M_k}\|\phi\|_{L^\infty(\mathrm{supp}\phi)}\|f_k\|_{L^p(B_1)}|B_{r_kR}(x_k)|^{1-\frac{1}{p}}\leq c\|\phi\|_{L^\infty(\mathrm{supp}\phi)}r_k^{1-\frac{n}{p}-\alpha}k^{-1}\to0
\end{equation*}
since $\alpha\leq 1-n/p$. In order to estimate $T^2_k,T^3_k$ we do the same preliminary remark: given a constant vector $b\in\R^n$, and integrating by parts one has
\begin{equation*}
\int_{\mathrm{supp}\phi}b\cdot\nabla\phi=0.
\end{equation*}
Hence, taking $b=F_k(x_k)$,
\begin{eqnarray*}
\left|\int_{B_R}F_k(x_k+r_kx)\cdot\nabla\phi(x)dx\right|&=&\left|\int_{B_R}(F_k(x_k+r_kx)-F_k(x_k))\cdot\nabla\phi(x)dx\right|\\
&\leq& R^\alpha r_k^\alpha\|F_k\|_{C^{0,\alpha}(B_1)}\|\nabla\phi\|_{L^2(\mathrm{supp}\phi)}.
\end{eqnarray*}
Hence, since $\|F_k\|_{C^{0,\alpha}(B_1)}\leq cI_k\leq ck^{-1}M_k$,
\begin{equation*}
|T^2_k|\leq c\|\nabla\phi\|_{L^2(\mathrm{supp}\phi)}k^{-1}\to0.
\end{equation*}
Regarding $T^3_k$ we have to reason in two steps. First imagine that we are proving the present theorem in a weaker form; that is, given the $C^{0,\alpha'}$ regularity of $F,A$ for some $\alpha'\in (0,1)$ with $\alpha'\leq1-n/p$, we are proving $C^{1,\alpha}$ local regularity with $0<\alpha<\alpha'$. Taking $b=A_k(x_k)\nabla u_k(x_k)$ and reasoning as before we get
\begin{eqnarray*}
\left|\int_{B_R}A_k(x_k+r_kx)\nabla u_k(x_k)\cdot\nabla\phi(x)dx\right|&=&\left|\int_{B_R}(A_k(x_k+r_kx)-A_k(x_k))\nabla u_k(x_k)\cdot\nabla\phi(x)dx\right|\\
&\leq& R^{\alpha'} r_k^{\alpha'}\|A_k\|_{C^{0,\alpha'}(B_1)}|\nabla u_k(x_k)|\cdot\|\nabla\phi\|_{L^2(\mathrm{supp}\phi)}.
\end{eqnarray*}
Now, using that $\|A_k\|_{C^{0,\alpha'}(B_1)}\leq\overline L$, the fact that $|\eta(x_k)\nabla u_k(x_k)|\leq 2M_k$ by \eqref{boundgradient} and $\alpha'>\alpha$, then
\begin{equation*}
|T^3_k|\leq c\|\nabla\phi\|_{L^2(\mathrm{supp}\phi)}r_k^{\alpha'-\alpha}\to0.
\end{equation*}
Then, once the result is proved with the suboptimal requirement $\alpha<\alpha'$, this means that we have in particular the a priori local $L^\infty$-bound for the gradient of the solutions; that is,
\begin{equation}
|\nabla u_k(x_k)|\leq cI_k\leq c\frac{M_k}{k}.
\end{equation}
Then, reasoning as before we get
\begin{eqnarray*}
\left|\int_{B_R}A_k(x_k+r_kx)\nabla u_k(x_k)\cdot\nabla\phi(x)dx\right|&=&\left|\int_{B_R}(A_k(x_k+r_kx)-A_k(x_k))\nabla u_k(x_k)\cdot\nabla\phi(x)dx\right|\\
&\leq& R^{\alpha} r_k^{\alpha}\|A_k\|_{C^{0,\alpha}(B_1)}|\nabla u_k(x_k)|\cdot\|\nabla\phi\|_{L^2(\mathrm{supp}\phi)}.
\end{eqnarray*}
Now, using that $\|A_k\|_{C^{0,\alpha}(B_1)}\leq\overline L$ and the fact that $|\nabla u_k(x_k)|\leq ck^{-1}M_k$, then
\begin{equation*}
|T^3_k|\leq c\|\nabla\phi\|_{L^2(\mathrm{supp}\phi)}k^{-1}\to0.
\end{equation*}
This in particular means that there exists a sequence $\delta_k\to0$ such that
\begin{equation*}
\int_{\mathrm{supp}\phi}A_k(x_k+r_kx)\nabla w_k(x)\cdot\nabla\phi(x)\leq \delta_k(\|\phi\|_{L^\infty(\mathrm{supp}\phi)}+\|\nabla\phi\|_{L^2(\mathrm{supp}\phi)}).
\end{equation*}
Then, arguing as in the proof of Theorem \ref{t:C0alphaPriori}, we can conclude that, up to further subsequences, the convergence $w_k\to v$ is also weak in $H^1_\loc(\R^n)$ (i.e. on any compact of $\R^n$). Thus, $v\in H^1_\loc(\R^n)$. Moreover $v$ is entire $\overline A$-harmonic in $\R^n$; that is, for any $\phi\in C^\infty_c(\R^n)$
\begin{equation*}
\int_{\R^n}\overline A\nabla v\cdot\nabla\phi=0.
\end{equation*}

\noindent{\bf Step 5: the conclusion by the Liouville theorem.}
Summing up all the information obtained, we have a limiting profile $v$ which is $H^1_\loc(\R^n)$, and entire $\overline A$-harmonic. Moreover $v$ is not constant with non constant gradient since $|\nabla v(0)-\nabla v(\xi)|> 1/4$ and its growth is subquadratic by \eqref{subquadratic}. Then, this is in contradiction with the Liouville Theorem \ref{t:liouville}.
\end{proof}

\section{A posteriori $C^{0,\alpha}$ and $C^{1,\alpha}$ estimates, $C^{k,\alpha}$ estimates}\label{s:6}

\subsection{A posteriori $C^{0,\alpha}$ and $C^{1,\alpha}$ estimates}
In this section we prove that the regularity (together with the estimates) in the previous section is enjoyed a posteriori by any weak solution of \eqref{E} with the suitable assumptions on the data. This is done by a regularization-approximation scheme. The latter works as follows:

\begin{itemize}
\item[1)] {\bf Regularization:} one regularizes the data $A,f,F$ by convolution with a standard family of mollifiers depending on a small parameter $\varepsilon>0$;
\item[2)] {\bf Approximation:} one defines a family of $\varepsilon$-regularized problems with data $A_\varepsilon,f_\varepsilon,F_\varepsilon$. Given a particular solution $u$ of \eqref{E}, solving the associated Dirichlet problem for the $\varepsilon$-equation with boundary data $u$ prescribed on the boundary of a ball, one can imply suitable convergence of the unique solution $u_\varepsilon$ to $u$;
\item[3)] {\bf A posteriori estimates:} by Corollary \ref{c:smooth}, $u_\varepsilon$ are smooth and hence uniform in $\varepsilon$ estimates are available by Theorems \ref{t:C0alphaPriori} and \ref{t:C1alphaPriori} and pass to the limit $u$, finally providing
\end{itemize}

\begin{Theorem}[A posteriori $C^{0,\alpha}$ estimates]\label{t:C0alphaPosteriori}
Let $p>n/2$, $q>n$, $0<r<R$. Let $\alpha\in(0,1)$ such that
\begin{equation*}
\alpha\leq\min\{2-n/p,1-n/q\}.
\end{equation*}
Let $A\in C^{0,\omega(\cdot)}(B_R)$ with $\|A\|_{C^{0,\omega(\cdot)}(B_R)}\leq\overline L$ and $\omega$ is any given modulus of continuity.
Then, there exists a constant $C>0$ depending only on $n,p,q,\alpha,r,R$, the ellipticity constants in $B_R$ and $\overline L$ such that
\begin{equation*}
\|u\|_{C^{0,\alpha}(B_{r})}\leq C(\|u\|_{L^2(B_R)}+\|f\|_{L^p(B_R)}+\|F\|_{L^q(B_R)})
\end{equation*}
for any weak solution of \eqref{E} in $B_R$.
\end{Theorem}

\begin{Theorem}[A posteriori $C^{1,\alpha}$ estimates]\label{t:C1alphaPosteriori}
Let $p>n$, $0<r<R$. Let $\alpha\in(0,1)$ such that
\begin{equation*}
\alpha\leq 1-n/p.
\end{equation*}
Let $A\in C^{0,\alpha}(B_R)$ with $\|A\|_{C^{0,\alpha}(B_R)}\leq\overline L$.
Then, there exists a constant $C>0$ depending only on $n,p,\alpha,r,R$, the ellipticity constants in $B_R$ and $\overline L$ such that
\begin{equation*}
\|u\|_{C^{1,\alpha}(B_{r})}\leq C(\|u\|_{L^2(B_R)}+\|f\|_{L^p(B_R)}+\|F\|_{C^{0,\alpha}(B_R)})
\end{equation*}
for any weak solution of \eqref{E} in $B_R$.
\end{Theorem}

\begin{proof}[Proof of Theorem \ref{t:C0alphaPosteriori}]
The proof can be divided into three main steps.\\

\noindent {\bf Step 1: Regularization.} Let us define a standard family of mollifiers: given $\eta\in C^\infty_c(\R^n)$ with $\int_{\R^n}\eta=1$ and $\mathrm{supp}\eta=B_1$, $\eta$ radially decreasing and $\eta\geq0$. Then, given $\varepsilon>0$, one defines
\begin{equation*}
\eta_\varepsilon(x)=\frac{1}{\varepsilon^{n}}\eta\left(\frac{x}{\varepsilon}\right).
\end{equation*}
so that $\mathrm{supp}\eta_\varepsilon=\varepsilon\mathrm{supp}\eta=B_\varepsilon$ and $\int_{\R^n}\eta_\varepsilon=1$. Let us define the regularized data by convolution with the mollifiers $a_{ij}^\varepsilon=a_{ij}*\eta_\varepsilon$ (so that $A_\varepsilon=(a_{ij}^\varepsilon)$), $f_\varepsilon=f*\eta_\varepsilon$ and $F_\varepsilon=F*\eta_\varepsilon$. These new data are well defined in $B_R$ for any $0<R<1$ provided that $0<\varepsilon\leq\overline\varepsilon(R)$ and they are smooth in $\overline{B_R}$. Moreover, if the original data $g\in L^p(B_1)$, then $\|g_\varepsilon\|_{L^p(B_R)}\leq \|g\|_{L^p(B_1)}$ and $\|g_\varepsilon-g\|_{L^p(B_R)}\to0$. Moreover, if the original data $g\in C^{0,\omega(\cdot)}(B_1)$ for some modulus of continuity $\omega$, then $\|g_\varepsilon\|_{C^{0,\omega(\cdot)}(B_R)}\leq\|g\|_{C^{0,\omega(\cdot)}(B_1)}$ and $\|g_\varepsilon-g\|_{L^\infty(B_R)}\to0$.\\

\noindent {\bf Step 2: Approximation.}
Given $u\in H^1(B_1)$ a weak solution to \eqref{E} in $B_1$, let us define the problem
\begin{equation}\label{equeps}
\begin{cases}
-\mathrm{div}(A_\varepsilon\nabla v)=f_\varepsilon+\mathrm{div}F_\varepsilon & \mathrm{in \ }B_{3/4}\\
v=u & \mathrm{on \ }\partial B_{3/4}.
\end{cases}
\end{equation}
Then $v\in H^1(B_{3/4})$ is a solution of the above Dirichlet problem if and only if $w=v-u\in H^{1}_0(B_{3/4})$ is a solution to
\begin{equation}\label{eqweps}
\begin{cases}
-\mathrm{div}(A_\varepsilon\nabla w)=f_\varepsilon+\mathrm{div}F_\varepsilon +\mathrm{div}(A_\varepsilon\nabla u) & \mathrm{in \ }B_{3/4}\\
w=0 & \mathrm{on \ }\partial B_{3/4}.
\end{cases}
\end{equation}
By the Lax-Milgram theorem, fixed $\varepsilon>0$ there exists unique solution $w_\varepsilon\in H^{1}_0(B_{3/4})$ to \eqref{eqweps}. This is true since
\begin{equation*}
\langle w,\phi\rangle_{H^1_0(B_{3/4})}^\varepsilon:=\int_{B_{3/4}}A_\varepsilon\nabla w\cdot\nabla\phi
\end{equation*}
defines a bilinear form in the Hilbert space $H^1_0(B_{3/4})$ which is coercive and continuous. Moreover
\begin{equation*}
L_\varepsilon(\phi):=\int_{B_{3/4}}f_\varepsilon\phi-(F_\varepsilon+A_\varepsilon\nabla u)\cdot\nabla\phi
\end{equation*}
is a linear and continuous functional. Hence
\begin{eqnarray*}
\|w_\varepsilon\|_{H^1_0(B_{3/4})}&\leq&\|L_\varepsilon\|_{(H^1_0(B_{3/4}))^*}=\sup_{\|\phi\|_{H^1_0(B_{3/4})}=1}|L_\varepsilon(\phi)|\\
&\leq& \tilde c(\|f_\varepsilon\|_{L^p(B_{3/4})}+\|F_\varepsilon\|_{L^q(B_{3/4})}+\|\nabla u\|_{L^2(B_{3/4})})\\
&\leq& c (\|u\|_{L^2(B_{1})}+\|f\|_{L^p(B_{1})}+\|F\|_{L^q(B_{1})}),
\end{eqnarray*}
where $\tilde c>0$ depends on $\|A_\varepsilon\|_{L^\infty(B_{3/4})}$ and $c>0$ depends on $\|A\|_{L^\infty(B_{1})}\leq L$. Then, there exists $w\in H^1_0(B_{3/4})$ such that, up to subsequences, $w_\varepsilon$ weakly converges to $w$. Then, it is easy to see that the equations for $w_\varepsilon$ pass to the limit giving that $w$ is the unique solution to $-\mathrm{div}(A\nabla w)=0$ in $H^{1}_0(B_{3/4})$. This implies that $w\equiv0$. Then, testing the equation of the $w_\varepsilon$ with $w_\varepsilon$ itself and passing to the limit one can infer that the convergence $w_\varepsilon\to w\equiv0$ is also strong in $H^1_0(B_{3/4})$. Then, the unique solution $u_\varepsilon=w_\varepsilon+u$ to \eqref{equeps} strongly converges in $H^1(B_{3/4})$ to $u$ and $\|u_\varepsilon\|_{L^2(B_{3/4})}\leq \|w_\varepsilon\|_{L^2(B_{3/4})}+\|u\|_{L^2(B_{3/4})}\leq 2\|u\|_{L^2(B_{1})}$.\\

\noindent {\bf Step 3: A posteriori estimates.} Thanks to Corollary \ref{c:smooth} we have that $u_\varepsilon\in C^\infty(B_{3/4})$. Hence, we can apply on this family of regularized solutions the a priori estimates in Theorem \ref{t:C0alphaPriori}; that is, there exists a constant $c>0$ not depending on $\varepsilon>0$ such that
\begin{eqnarray*}
\|u_\varepsilon\|_{C^{0,\alpha}(B_{2/3})}&\leq& c(\|u_\varepsilon\|_{L^2(B_{3/4})}+\|f_\varepsilon\|_{L^p(B_{3/4})}+\|F_\varepsilon\|_{L^q(B_{3/4})})\\
&\leq& c(\|u\|_{L^2(B_{1})}+\|f\|_{L^p(B_{1})}+\|F\|_{L^q(B_{1})}).
\end{eqnarray*}
The uniform bound in $C^{0,\alpha}(B_{2/3})$ allows to have uniform convergence $u_\varepsilon\to u$ in $\overline{B_{2/3}}$ by the Ascoli-Arzel\'a theorem, giving in particular that for $x,y\in B_{1/2}$ with $x\neq y$
$$|u_\varepsilon(x)|+\frac{|u_\varepsilon(x)-u_\varepsilon(y)|}{|x-y|^\alpha}\to |u(x)|+\frac{|u(x)-u(y)|}{|x-y|^\alpha}.$$
Then
$$|u(x)|+\frac{|u(x)-u(y)|}{|x-y|^\alpha}\leq c (\|u\|_{L^2(B_{1})}+\|f\|_{L^p(B_{1})}+\|F\|_{L^q(B_{1})}),$$
and passing to the supremum we get the estimate.
\end{proof}

\begin{proof}[Proof of Theorem \ref{t:C1alphaPosteriori}]
The {\bf Steps 1,2} are as in the proof of the previous result (just changing the norm of the field term which is now in $C^{0,\alpha}$).\\

\noindent {\bf Step 3: A posteriori estimates.} Thanks to Corollary \ref{c:smooth} we have that $u_\varepsilon\in C^\infty(B_{3/4})$. Hence, we can apply on this family of regularized solutions the a priori estimates in Theorem \ref{t:C1alphaPriori}; that is, there exists a constant $c>0$ not depending on $\varepsilon>0$ such that
\begin{eqnarray*}
\|u_\varepsilon\|_{C^{1,\alpha}(B_{2/3})}&\leq& c(\|u_\varepsilon\|_{L^2(B_{3/4})}+\|f_\varepsilon\|_{L^p(B_{3/4})}+\|F_\varepsilon\|_{C^{0,\alpha}(B_{3/4})})\\
&\leq& c(\|u\|_{L^2(B_{1})}+\|f\|_{L^p(B_{1})}+\|F\|_{C^{0,\alpha}(B_{1})}).
\end{eqnarray*}
The uniform bound in $C^{1,\alpha}(B_{2/3})$ allows to have uniform convergence $u_\varepsilon\to u$ and $\nabla u_\varepsilon\to\nabla u$ in $\overline{B_{2/3}}$ by the Ascoli-Arzel\'a theorem, giving in particular that for $x,y\in B_{1/2}$ with $x\neq y$
$$|u_\varepsilon(x)|+|\nabla u_\varepsilon(x)|+\frac{|\nabla u_\varepsilon(x)-\nabla u_\varepsilon(y)|}{|x-y|^\alpha}\to |u(x)|+|\nabla u(x)|+\frac{|\nabla u(x)-\nabla u(y)|}{|x-y|^\alpha}.$$
Then
$$|u(x)|+|\nabla u(x)|+\frac{|\nabla u(x)-\nabla u(y)|}{|x-y|^\alpha}\leq c (\|u\|_{L^2(B_{1})}+\|f\|_{L^p(B_{1})}+\|F\|_{C^{0,\alpha}(B_{1})}),$$
and passing to the supremum we get the estimate.
\end{proof}

\subsection{$C^{k,\alpha}$ estimates}
The $C^{1,\alpha}$ estimate can be iterated on partial derivatives, and this implies

\begin{Theorem}[$C^{k,\alpha}$ estimates]\label{t:Ckalpha}
Let $\alpha\in(0,1)$, $k\geq2$, $0<r<R$. Let $A\in C^{k-1,\alpha}(B_R)$ with $\|A\|_{C^{k-1,\alpha}(B_R)}\leq\overline L$.
Then, there exists a constant $C>0$ depending only on $n,\alpha,k,r,R$, the ellipticity constants in $B_R$ and $\overline L$ such that
\begin{equation*}
\|u\|_{C^{k,\alpha}(B_{r})}\leq C(\|u\|_{L^2(B_R)}+\|f\|_{C^{k-2,\alpha}(B_R)}+\|F\|_{C^{k-1,\alpha}(B_R)})
\end{equation*}
for any weak solution of \eqref{E} in $B_R$.
\end{Theorem}

\begin{proof}
We reason by induction on $k\geq2$. Let us fix wlog $r=1/2$ and $R=1$. Let us assume $k=2$. In these conditions we already know by Theorem \ref{t:C1alphaPosteriori} that $u\in C^{1,\beta}_\loc(B_1)$ for any $\beta\in(0,1)$. Then $u_i=\partial_iu$ solves for any $0<r<1$
\begin{equation*}
-\mathrm{div}(A\nabla u_i)=\mathrm{div}(\partial_iA\nabla u+fe_i+\partial_iF)\qquad\mathrm{in \ }B_r.
\end{equation*}
Then, by Theorem \ref{t:C1alphaPosteriori} $u_i\in C^{1,\alpha}_\loc(B_r)$ with
\begin{eqnarray*}
\|u_i\|_{C^{1,\alpha}(B_{1/2})}&\leq& c(\|u_i\|_{L^2(B_{2/3})}+\|\partial_iA\nabla u+fe_i+\partial_iF\|_{C^{0,\alpha}(B_{2/3})})\\
&\leq& c(\|u\|_{L^2(B_{1})}+\|A\|_{C^{1,\alpha}(B_{1})}\|u\|_{C^{1,\alpha}(B_{2/3})}+\|f\|_{C^{0,\alpha}(B_{1})}+\|F\|_{C^{1,\alpha}(B_{1})}).
\end{eqnarray*}
Then, applying again the $C^{1,\alpha}$-estimate of $u$ from $B_1$ to $B_{2/3}$ in the last line we have the desired estimate. Then, supposing the result true for $k\geq2$ and proving it for $k+1$ follows the same kind of argument. Just notice that for any $k\in\mathbb N$ and $\alpha\in(0,1]$ one has $\|fg\|_{C^{k,\alpha}}\leq\|f\|_{C^{k,\alpha}}\|g\|_{C^{k,\alpha}}$.
\end{proof}

\section*{Acknowledgment}
I would like to express my gratitude to Gabriele Cora, Gabriele Fioravanti, Susanna Terracini, and Giorgio Tortone, whose contributions and collaboration over the past years have helped shape this work. I would also like to thank Timoth\'e Lemistre for the 2D examples in Remark \ref{r:superpolynomial}. Finally, I would like to thank Dario Mazzoleni, who taught this PhD course with me, and all the students who attended it. 

The author is research fellow of Istituto Nazionale di Alta Matematica INDAM group GNAMPA and supported by the GNAMPA project E5324001950001 \emph{PDE ellittiche che degenerano su variet\`a di dimensione bassa e frontiere libere molto sottili}. He is also supported by the PRIN project 2022R537CS \emph{$NO^3$ - Nodal Optimization, NOnlinear elliptic equations, NOnlocal geometric problems, with a focus on regularity}.


\end{document}